\documentclass[11pt]{article}
\pdfoutput=1
\usepackage[margin=1in]{geometry}
\usepackage{setspace}
\usepackage{graphicx} 
\usepackage{amsmath,amssymb,amsthm,mathrsfs,enumerate}
\usepackage[numbers]{natbib}
\usepackage{bm}
\usepackage{verbatim}
\usepackage{mathtools}
\usepackage{tikz}
\usepackage{titlesec}
\usepackage[titletoc,toc,title]{appendix}
\AtBeginEnvironment{appendices}

\usetikzlibrary{arrows.meta}
\RequirePackage[colorlinks,citecolor=blue,urlcolor=blue]{hyperref}
\newcommand{\dd}{\mathop{}\!\mathrm{d}}
\newcommand{\R}{\mathbb{R}}
\newtheorem{theorem}{Theorem}[section]
\newtheorem{lemma}[theorem]{Lemma}
\newtheorem{proposition}[theorem]{Proposition}

\newtheorem{remark}[theorem]{Remark}
\theoremstyle{definition}
\numberwithin{equation}{section}
\newtheorem{definition}[theorem]{Definition}

\newcommand{\E}{\mathbb{E}}
\providecommand{\abs}[1]{\lvert#1\rvert}

\title{\bf Optimal Two-Mode Switching with Setup Costs for a Controlled Ornstein--Uhlenbeck Process}
\author{Hongkun Xia\footnote{Academy of Mathematics and Systems Science, Chinese Academy of Sciences, Beijing, 100190, China; School of Mathematical Sciences, University of Chinese Academy of Sciences, Beijing, 100049, China; xiahongkun@amss.ac.cn},\ \ 
Dacheng Yao\footnote{State Key Laboratory of Mathematical Sciences, Academy of Mathematics and Systems Science, Chinese Academy of Sciences, Beijing 100190, China; School of Mathematical Sciences, University of Chinese Academy of Sciences, Beijing, 100049, China; dachengyao@amss.ac.cn}}

\begin{document}
\maketitle

\begin{abstract}
    We study an infinite-horizon discounted optimal switching problem for an Ornstein--Uhlenbeck process whose mean-reversion level can be switched between two values. Switching incurs positive, potentially asymmetric fixed costs, and the mode with the higher mean-reversion level entails an additional running cost. The revenue rate is proportional to the positive part of the state. We characterize the value function and identify parameter conditions under which an optimal policy either remains permanently in one mode after a possible costly initial switch or follows a threshold rule. The threshold formulation allows infinite boundaries and therefore includes one-way switching. A key technical contribution is to establish existence and uniqueness of the relevant solutions to the coupled nonlinear free-boundary equations arising from value-matching and smooth-fit conditions. In the corresponding parameter regimes, these solutions determine the finite switching thresholds and the coefficients of the candidate value functions. A verification argument then establishes the optimality of the resulting policies over all admissible switching strategies.
\end{abstract}

\textbf{Keywords:} optimal switching, Ornstein--Uhlenbeck process, infinite-horizon, discounted model

\section{Introduction}
Mean-reverting diffusion models arise in a variety of applications, including production and inventory management \cite{cadenillas2010optimal,cadenillas2013optimal}, asset trading \cite{EloeLYYZ2008,ngo2016optimal,Suzuki2021,ZhangZhang2008},
and customer satisfaction management \cite{KanoriaLobelLu2024}. In some applications, the operating mode affects not only the current
payoff but also the evolution of the underlying state. In particular, changing the operating mode may alter the level toward which the state reverts. Such decisions involve a dynamic tradeoff: a mode with a higher mean-reversion level may improve future revenues while requiring a higher current operating cost. Furthermore, fixed costs of changing modes introduce an additional incentive to postpone adjustments and retain the current mode.

In this paper, we study an infinite-horizon discounted optimal switching problem in which the mean-reversion level of an Ornstein--Uhlenbeck (OU) process can be switched between two values. Switches between two modes incur positive, potentially asymmetric fixed costs. The revenue rate is proportional to the positive part of the state, while the mode with the higher mean-reversion level incurs an additional running cost. Our objective is to characterize the value function and determine how the model parameters govern the structure of an optimal switching policy. 

More precisely, let the state process $X=\{X(t):t\geq 0\}$ satisfy
\begin{equation*}
\mathrm{d}X(t)=\theta\bigl(\mu(t)-X(t)\bigr)\,\mathrm{d}t+\sigma\,\mathrm{d}B(t),\qquad t\geq 0,
\end{equation*}
where $\theta>0$ is the mean-reversion speed, $\sigma>0$ is the volatility parameter, and $B$ is a standard $\{\mathcal{F}_t\}$-Brownian motion on a filtered probability space $(\Omega,\mathcal{F},\{\mathcal{F}_t\}_{t\geq 0},\mathbb{P})$ satisfying the usual conditions. At any time $t\geq0$, the control $\mu(t)$ takes values in $\mathcal{U}:=\{\mu_{\ell},\mu_h\}$, $\mu_{\ell}<\mu_h$, and represents the mean-reversion level associated with the active operating mode. Thus, switching changes the drift of the state process. We assume a common positive volatility in the two operating modes. This simplifying assumption is common in the literature; see e.g. \cite{CaoYao2018,OrmeciMatogluVandeVate2011}. An upward switch from $\mu_{\ell}$ to $\mu_h$ incurs a fixed cost $K_h>0$, and a downward switch from $\mu_h$ to $\mu_{\ell}$ incurs a fixed cost $K_{\ell}>0$. Write
\[
K(\mu_{\ell},\mu_h)=K_h\quad\text{and}\quad K(\mu_h,\mu_{\ell})=K_{\ell}.
\]
The running revenue at state $x$ is $rx^+$, where $r>0$ and $x^+:=\max\{x,0\}$. We normalize the running cost of the low mode to zero and denote the additional running cost of the high mode by $c>0$.

Given an initial state $X(0)=x\in\mathbb{R}$ and an initial mode $\mu(0-)=\mu\in\mathcal{U}$, a switching policy is represented by $\pi=\{\tau_n:n\geq1\}$, where $\tau_n$ is an $\{\mathcal{F}_t\}$-stopping time taking values in $[0,\infty]$. Here we note $X(0)=X(0-)$ because there is no jump for the state. We require that the switching times be nondecreasing, that $\tau_n<\tau_{n+1}$ on $\{\tau_n<\infty\}$, and that $\tau_n\to\infty$ almost surely. Each finite $\tau_n$ represents an actual switch to the other mode. In particular, $\tau_1=0$ represents an immediate switch from the given initial mode. Since there are only two modes, the initial mode and the switching times uniquely determine all subsequent modes.
After the last actual switch, all subsequent switching times are set to infinity. The resulting control process is adapted, right-continuous, and piecewise constant, and is given by
\[
\mu(t)
=\mu\mathbf1_{\{0\leq t<\tau_1\}}
+\sum_{\{n\geq1:\tau_n<\infty\}}
\mu(\tau_n)\mathbf1_{\{\tau_n\leq t<\tau_{n+1}\}},
\qquad t\geq0.
\]
Such a policy is called admissible, and $\Pi$ denotes the class of all admissible policies.

For a discount rate $\beta>0$, the expected discounted net profit under policy $\pi$ is
\begin{align}
\mathcal{J}\bigl((x,\mu),\pi\bigr)
&= \mathbb{E}_{(x,\mu)}\Big[
    \int_0^\infty e^{-\beta t}
    \Bigl(
        r\bigl(X(t)\bigr)^+
        - c\mathbf{1}_{\{\mu(t)=\mu_h\}}
    \Bigr)\,\dd t
\nonumber  \\
&\qquad {}-
    \sum_{\{n:\tau_n<\infty\}}
    e^{-\beta\tau_n}
    K\bigl(\mu(\tau_n-),\mu(\tau_n)\bigr)
    \Big],\label{eq:objective}
\end{align}
where $\mathbb{E}_{(x,\mu)}$ denotes expectation under the specified initial state and mode. The objective is to maximize the expected discounted net profit:
\[
V(x,\mu):=\sup_{\pi\in\Pi}\mathcal J((x,\mu),\pi).
\]

An application of this model is short-video advertising, where higher intensity broadens targeting from users with strong product affinity to those with moderate affinity. Here, $X(t)$ denotes goodwill, and $\mu_\ell,\mu_h$ are its equilibrium levels under the two advertising intensities. The dynamics capture goodwill accumulation and proportional depreciation, following Nerlove and Arrow \cite{NerloveArrow1962}, with additive Brownian noise as in \cite{Marinelli2007,Raman2006}. Our revenue specification $r(X(t))^+$ is consistent with that in \cite{GrossetViscolani2009}: revenue is proportional to goodwill when goodwill is positive and zero otherwise. Higher intensity incurs an additional running cost $c$, while switching costs represent campaign adjustment expenses.

Our problem belongs to the literature on stochastic optimal switching, which has been studied for Brownian models \cite{Doshi1978,OrmeciMatogluVandeVate2011,WuChao2014}, geometric Brownian motion (GBM) \cite{BrekkeOksendal1994,Dixit1989,duckworth2001model,LyVathPham2007}, mean-reverting diffusion \cite{ngo2016optimal,Suzuki2016,Suzuki2021}, and general diffusion \cite{bayraktar2010one,Doshi1978,Pham2007,pham2009optimal,Song2012}. 
Related optimal stopping and control problems under exogenous regime switching are studied in \cite{GuoZhang2004,SongStockbridgeZhu2011,Zhu2011}.
A distinction relevant to the present paper is whether the state dynamics depend on the active mode. In models with \emph{mode-independent dynamics}, switching affects the payoff structure but leaves the state dynamics unchanged. In models with \emph{mode-dependent dynamics}, the active mode also enters the drift or diffusion coefficient. Our model belongs to the latter category: its mean-reverting dynamics are endogenous to the switching decisions.

Several analytical foundations for switching problems with mode-dependent dynamics have been established. For example, Pham \cite{Pham2007} proves the smooth-fit property of the value functions for one-dimensional optimal switching problems using a viscosity solution approach. Ly Vath and Pham \cite{LyVathPham2007} combine viscosity solutions and smooth fit to obtain explicit solutions in two-mode geometric Brownian settings, including a case with different diffusion operators and identical running profit functions. Bayraktar and Egami \cite{bayraktar2010one} develop a framework for one-dimensional diffusions based on dynamic programming and coupled optimal stopping problems. Their general formulation allows mode-dependent dynamics, whereas several of their subsequent characterizations of switching regions assume mode-independent dynamics. These results provide important foundations for the analysis of switching problems. Identifying the policy structure and resolving the associated free-boundary equations nevertheless require further analysis for the particular dynamics and payoff functions under consideration.

A related controlled OU model is studied by Kanoria et al. \cite{KanoriaLobelLu2024} in the context of customer churn management. Their objective is to maximize rewards over a customer's random lifetime, and the service mode determines the dynamics of customer satisfaction. They characterize optimal policies in a base model with zero volatility in one mode and no fixed switching costs, and also
examine fixed switching costs numerically in an extension. Our setting instead features a positive volatility in both modes and positive, potentially asymmetric fixed switching costs. Our focus is on an analytical characterization of the resulting policy structures and switching boundaries.

Our main contribution is an analytical characterization of the value function and optimal switching policies for an OU process with a controlled mean-reversion level and positive, potentially asymmetric fixed switching costs. We identify conditions on the model parameters under which no-switching policies or $(d,u)$-threshold policies are optimal. Here, a no-switching policy means permanent operation in a single mode, possibly after a switch at time zero whose fixed cost is included in the objective. Infinite thresholds are allowed, so the threshold formulation also includes one-way switching. Specifically, we find that the optimal policy classification is determined by the profit difference between the two permanent-mode policies. Bayraktar and Egami \cite{bayraktar2010one} also use profit difference between the two permanent-mode policies in their Section~2.5. They provide a sufficient condition for retaining the initial mode forever under mode-dependent dynamics, whereas their characterizations of other no-switching policies and threshold policies assume mode-independent dynamics. Our analysis retains mode dependence throughout the complete policy classification and threshold analysis.

A technical difficulty in the verification approach is to establish the existence of solutions to the nonlinear free-boundary equations arising from value-matching and smooth-fit conditions; see, e.g., the introduction in \cite{bayraktar2010one}. In our model, these conditions yield a coupled nonlinear system for the switching thresholds and the coefficients of the candidate value functions associated with the two mode-dependent OU generators. We establish existence and uniqueness of the relevant solutions in the corresponding parameter regimes, thereby rigorously determining the finite switching boundaries and the associated candidate value functions. Ly Vath and Pham \cite{LyVathPham2007} likewise need to solve a similar system of free-boundary equations in their Section~4.1. Their GBM model gives the free-boundary equations a power-function structure that facilitates direct algebraic analysis. In our OU model, the free-boundary equations involve more intricate integral representations. We therefore analyze the shape and zeros of an auxiliary function derived from the free-boundary equations, and combine these properties with the integral of its positive part and the value-matching conditions to establish existence and uniqueness of the relevant solutions. We then prove that these functions satisfy the inequalities required for verification and establish the optimality of the resulting policies over the full class of admissible switching strategies. 

The rest of the paper is organized as follows. Section~\ref{sec:policies} introduces the two policy classes and characterizes their expected discounted profits. Section~\ref{sec:mainresults} presents the main results on the value function and optimal switching policies.
Section~\ref{sec:parameters} establishes existence and uniqueness of the relevant switching thresholds, and Section~\ref{sec:optimality} verifies the optimality of the proposed policies. Section~\ref{sec:conclusion} concludes the paper.

We close this section with some frequently used notation. Let $\Phi$ and $\phi$ be the cumulative distribution function (CDF) and the probability density function (PDF) of the standard normal distribution, respectively. Let $\mathcal{C}^n(X)$ denote the space of $n$-times continuously differentiable functions on $X$. Let $\mathcal{L}^1(X)$ denote the space of Lebesgue integrable functions on $X$. 

\section{Two Classes of Policies}

\label{sec:policies}

In this section, we introduce two classes of admissible policies, namely the no-switching policies and the $(d,u)$-threshold policies, and provide the profit function under these policies. These two classes provide the candidate policies for the optimization problem.

\subsection{No-Switching Policies}

We define the no-switching policies as follows.
\begin{definition}[No-Switching Policies]
\label{def:no-switching}
An admissible policy $\pi$ is called a no-switching policy if no switches occur strictly after the initial time for any initial state and mode $(X(0),\mu(0-))\in \mathbb{R}\times \mathcal{U}$. 
\end{definition}

Among these policies, we consider the two fixed-mode policies, corresponding to the permanent adoption of either mode $\mu_h$ or mode $\mu_\ell$ for all $t > 0$. We emphasize that a policy is still classified as ``no-switching'' even if it involves an instantaneous switch at $t=0$ (e.g., switching from an initial mode $\mu_\ell$ to a permanent mode $\mu_h$), which naturally incurs the corresponding fixed switching cost upfront.

In what follows, we evaluate the expected discounted net profit under the two no-switching policies. Let $X^{h} = \{X^{h}(t):t\geq0\}$ and $X^{\ell} = \{X^{\ell}(t):t\geq0\}$ denote the controlled state processes under the no-switching policies with modes $\mu_h$ and $\mu_\ell$, respectively. That is, they satisfy the stochastic differential equations
\begin{equation}
\label{eq:Xh-Xl}
    \mathrm{d} X^{h}(t) = \theta\bigl(\mu_{h}-X^{h}(t)\bigr)\,\mathrm{d}t + \sigma \,\mathrm{d} B(t) \ \text{and} \ \mathrm{d} X^{\ell}(t) = \theta\bigl(\mu_{\ell}-X^{\ell}(t)\bigr)\,\mathrm{d}t + \sigma \,\mathrm{d} B(t).
\end{equation}
Let $J_h(x)$ (respectively, $J_{\ell}(x)$) denote the expected discounted net profit under the no-switching policy with mode $\mu_h$ (respectively, $\mu_\ell$), given the initial state $x$ and initial mode $\mu_h$ (respectively, $\mu_\ell$). That means no switching occurs for $t\geq0$. Thus, we have
\begin{equation}
    \label{eq:J-h-ell-0}
    \begin{aligned}
        J_{h}(x)
        &= \mathbb{E}_{(x,\mu_h)}\left[
        \int_{0}^{\infty}e^{-\beta t}
        \left(r(X^{h}(t))^{+}-c\right)\,\dd t
        \right] \quad\text{and}\\
        J_{\ell}(x)
        &= \mathbb{E}_{(x,\mu_\ell)}\left[
        \int_{0}^{\infty}e^{-\beta t}
        r(X^{\ell}(t))^{+}\,\dd t
        \right].
    \end{aligned}
\end{equation}
Define their difference 
\begin{equation}
    \label{eq:J}
    J(x) := J_{h}(x)-J_{\ell}(x).
\end{equation}
We next establish the analytical properties of functions $J_h(x)$, $J_{\ell}(x)$, and $J(x)$, which will be used in later sections.
Let $\Gamma_h$ and $\Gamma_\ell$ be the generators of $X^h$ and $X^\ell$ in \eqref{eq:Xh-Xl}, respectively, i.e.,
\begin{align}
    \Gamma_h f(x)= \frac{\sigma^2}{2}f''(x)+\theta(\mu_h-x) f'(x)\ \text{and}\
    \Gamma_\ell f(x)= \frac{\sigma^2}{2}f''(x)+\theta(\mu_\ell-x) f'(x).\label{eq:Gamma-h-ell}
\end{align}
Denote
\begin{equation}
\label{eq:m-v}
    m_t^h(x)=\mu_h+(x-\mu_h)e^{-\theta t},\quad 
    m_t^\ell(x)=\mu_\ell+(x-\mu_\ell)e^{-\theta t},\quad
    \text{and}\quad 
    v_t=\frac{\sigma^2}{2\theta}\left(1-e^{-2\theta t}\right).
\end{equation}

The proof of Lemma \ref{lem:J} is given in Appendix \ref{app-sub:lem-J}.

\begin{lemma}
    \label{lem:J}
    \begin{enumerate}[$(a)$]
        \item We have  
        \begin{align}
            J_h(x)&= r\int_{0}^{\infty}\int_{-\infty}^{m_t^h(x)} e^{-\beta t} \Phi\Big(\frac{z}{\sqrt{v_t}}\Big)\,\dd z\dd t -\frac{c}{\beta},\label{eq:J-h}\\
            J_{\ell}(x)&=  r\int_{0}^{\infty}\int_{-\infty}^{m_t^{\ell}(x)} e^{-\beta t} \Phi\Big(\frac{z}{\sqrt{v_t}}\Big)\,\dd z\dd t,\label{eq:J-l}
        \end{align}
    where $\Phi$ is the CDF of the standard normal distribution, and $m_t^h$, $m_t^\ell$, and $v_t$ are given in \eqref{eq:m-v}.  
    \item We have $J_{h},J_{\ell}\in \mathcal{C}^{2}(\R)\cap \mathcal{C}^{\infty}(\R \backslash \{0\})$; We have $J'_{h}(x)>0$ and $J_{\ell}'(x)>0$ for $x\in\R$, and both $J'_{h}(x)$ and $J_{\ell}'(x)$ are Lipschitz continuous on $\R$; 
    We have $J_{h}''(x)>0$, $J_{\ell}''(x)>0$ for $x\in\R$. 
    Moreover, $J_h(x)$ and $J_\ell(x)$ given by \eqref{eq:J-h}-\eqref{eq:J-l} satisfy the following differential equations: 
    \begin{align}
        &\Gamma_{h}J_h(x) - \beta J_h(x) + r x^+ - c = 0,\label{eq:ODE-J-h} \\
        &\Gamma_{\ell}J_{\ell}(x) - \beta J_\ell(x) + r x^+ = 0, \label{eq:ODE-J-l}
    \end{align}
    where $\Gamma_h$ and $\Gamma_\ell$ are the generators given in \eqref{eq:Gamma-h-ell}.
    \item We have the following limits
    \begin{align*}
        &\lim_{x\rightarrow-\infty}J_{h}(x)=-\frac{c}{\beta},\quad \lim_{x\rightarrow-\infty}J_{\ell}(x)=0,
        \quad \lim_{x\rightarrow\infty}J_{h}(x)=\lim_{x\rightarrow\infty}J_{\ell}(x)=\infty,\\
        &\lim_{x\to-\infty}J_{h}'(x)=\lim_{x\to-\infty}J_{\ell}'(x)=0,\qquad 
        \lim_{x\to\infty}J_{h}'(x)=\lim_{x\to\infty}J_{\ell}'(x)=\frac{r}{\beta+\theta}.
    \end{align*}
    \item We have 
    \begin{align}
     \label{eq:J-explicit}
        J(x)= r\int_{0}^{\infty}\int_{m_t^{\ell}(x)}^{m_t^h(x)} e^{-\beta t} \Phi\Big(\frac{z}{\sqrt{v_t}}\Big)\,\dd z\dd t -\frac{c}{\beta},
    \end{align}
    $J\in \mathcal{C}^{3}(\R)\cap \mathcal{C}^{\infty}(\R \backslash \{0\})$, $J'(x)>0$ for $x\in\R$, and the following limits: 
    \begin{align}
        &\lim_{x\to-\infty}J(x)=-\frac{c}{\beta},\quad
        \lim_{x\to\infty}J(x)=\frac{r\theta(\mu_{h}-\mu_{\ell})}{\beta(\beta+\theta)}-\frac{c}{\beta},\label{eq:lim-J}\\
        &\lim_{x\to-\infty}J'(x)=\lim_{x\to\infty}J'(x)=0.\nonumber
    \end{align}
    \end{enumerate}
\end{lemma}

\begin{remark}
\label{rem:lemma-1}
    \begin{enumerate}[$(a)$]
        \item It follows from \eqref{eq:ODE-J-h}-\eqref{eq:ODE-J-l} that we have the following ODEs
        \begin{align}
        &\Gamma_{\ell}J_{h}(x) - \beta J_h(x) + r x^+ - c=-\theta(\mu_{h}-\mu_{\ell})J_{h}'(x), \quad \text{and}
        \label{eq:ODE-J-h-Gamma ell} \\
        &\Gamma_{h}J_{\ell}(x) - \beta J_\ell(x) + r x^+ = \theta(\mu_{h}-\mu_{\ell})J_{\ell}'(x). \label{eq:ODE-J-l-Gamma h}
        \end{align}
        \item Subtracting \eqref{eq:ODE-J-l-Gamma h} from \eqref{eq:ODE-J-h} and subtracting \eqref{eq:ODE-J-l} from \eqref{eq:ODE-J-h-Gamma ell}, respectively, yields
        \begin{align}
               &\Gamma_{h}J(x) - \beta J(x)-c+\theta(\mu_{h}-\mu_{\ell})J_{\ell}'(x)=0,\quad \text{and}\label{eq:ODE-Jh-Jl-h}\\
               &\Gamma_{\ell}J(x) - \beta J(x)-c+\theta(\mu_{h}-\mu_{\ell})J_{h}'(x)=0.\label{eq:ODE-Jh-Jl-ell}
        \end{align}
    \end{enumerate}
\end{remark}

\subsection{\texorpdfstring{$(d,u)$}{(d,u)}-Threshold Policies}

Now we define the $(d,u)$-threshold policies.

\begin{definition}[$(d,u)$-Threshold Policies]
\label{def:du-policy}
An admissible policy $\pi$ is called a $(d,u)$-threshold policy with $-\infty \le d < u < \infty$ or $-\infty < d < u \le \infty$ if a switch from $\mu_h$ to $\mu_\ell$ is triggered exactly when the state process $X$ falls to or below $d$, and a switch from $\mu_\ell$ to $\mu_h$ is triggered exactly when the state process $X$ rises to or above $u$. 
\end{definition}

Notice that we naturally allow for degenerate single-threshold cases when $d = -\infty$ or $u = \infty$. For instance, under a $(d, \infty)$ policy, the system will never switch to the high mode $\mu_h$. Consequently: 
If the initial mode is $\mu(0-) = \mu_\ell$, the system remains in mode $\mu_\ell$ permanently (reducing to a no-switching policy);
If the initial mode is $\mu(0-) = \mu_h$, the policy forces an instantaneous switch to $\mu_\ell$ at time $t=0$ if the initial state $X(0) \le d$. Conversely, if $X(0)>d$, the system remains in mode $\mu_h$ until $\tau_1=\inf\{t>0:X(t)\le d\}$.
Since the nondegenerate mean-reverting OU process reaches every finite level almost surely, $\tau_1<\infty$ almost surely. At $\tau_1$, the system switches to $\mu_\ell$ and remains there indefinitely.

We next provide the expected profit function under a given $(d,\infty)$-threshold policy with $-\infty<d<\infty$, $(-\infty,u)$-threshold policy with $-\infty<u<\infty$, and $(d,u)$-threshold policy with $-\infty<d<u<\infty$, respectively. Before that, we first define four auxiliary functions and summarize some properties for these functions in the following lemma.
Denote
\begin{align}
    \psi_{+}^h(x)&=\int_{0}^{\infty}t^{\frac{\beta}{\theta}-1}\exp\Big(-\frac{t^{2}}{2}+\frac{\sqrt{2\theta}}{\sigma} (x-\mu_h)t\Big)\dd t,
    \label{eq:psi-h+}\\
    \psi_{-}^h(x)&=\int_{0}^{\infty}t^{\frac{\beta}{\theta}-1}\exp\Big(-\frac{t^{2}}{2}-\frac{\sqrt{2\theta}}{\sigma} (x-\mu_h)t\Big)\dd t,
     \label{eq:psi-h-}\\
    \psi_{+}^\ell(x)&=\int_{0}^{\infty}t^{\frac{\beta}{\theta}-1}\exp\Big(-\frac{t^{2}}{2}+\frac{\sqrt{2\theta}}{\sigma} (x-\mu_\ell)t\Big)\dd t,\quad \text{and} \label{eq:psi-l+}\\
    \psi_{-}^\ell(x)&=\int_{0}^{\infty}t^{\frac{\beta}{\theta}-1}\exp\Big(-\frac{t^{2}}{2}-\frac{\sqrt{2\theta}}{\sigma} (x-\mu_\ell)t\Big)\dd t. \label{eq:psi-l-}
\end{align}
All four integrals are finite. Moreover, differentiation under the integral sign is justified by dominated convergence, locally uniformly in $x$: near $t=0$, the integrands of the $n$th derivatives are bounded by a constant times $t^{\beta/\theta+n-1}$, while at infinity the Gaussian factor dominates every polynomial and linear exponential factor. Consequently, all four functions belong to
$\mathcal{C}^{\infty}(\mathbb{R})$.

The proof of Lemma \ref{lem:psi} is given in Appendix \ref{app-sub:lem-psi}.

\begin{lemma}
    \label{lem:psi}
    \begin{enumerate}[$(a)$]
        \item The functions $\psi_{+}^h(x)$ and $\psi_{-}^h(x)$ defined in \eqref{eq:psi-h+}-\eqref{eq:psi-h-} are two fundamental solutions of 
        \begin{align}
        \label{eq:ODE-psi-h}
            \Gamma_{h}\psi(x)-\beta \psi(x)=0,
        \end{align}
        Also, the functions $\psi_{+}^\ell(x)$ and $\psi_{-}^\ell(x)$ defined in \eqref{eq:psi-l+}-\eqref{eq:psi-l-} are two fundamental solutions of 
        \begin{align}
        \label{eq:ODE-psi-l}
           \Gamma_{\ell}\psi(x)-\beta \psi(x)=0.
        \end{align}
        \item  We have $(\psi_{-}^h)'(x)<0$, $(\psi_{+}^\ell)'(x)>0$, $(\psi_{-}^h)''(x)>0$, $(\psi_{+}^\ell)''(x)>0$ for $x\in\R$.
        \item We have the following limits:
        \begin{align}
            \lim_{x\to-\infty}\psi_-^h(x)&=\infty, \quad \lim_{x\to-\infty}(\psi_-^h)'(x)=-\infty, \label{eq:limit-psi-h-1}\\
            \lim_{x\to\infty} \psi_{-}^h(x)&=\lim_{x\to\infty} (\psi_{-}^h)'(x)=0, \label{eq:limit-psi-h-2}\\
             \lim_{x\to-\infty}\psi_+^{\ell}(x)&=\lim_{x\to-\infty}(\psi_+^{\ell})'(x)=0, \label{eq:limit-psi-ell-1}\\
            \lim_{x\to\infty} \psi_{+}^\ell(x)&=\lim_{x\to\infty} (\psi_{+}^\ell)'(x)=\infty,\label{eq:limit-psi-ell-2} 
        \end{align}
        \item $(\psi_-^h)'(x)$ is Lipschitz continuous on $(d,\infty)$ for any finite $d\in(-\infty,\infty)$ and $(\psi_+^{\ell})'(x)$ is Lipschitz continuous on $(-\infty,u)$ for any finite $u\in(-\infty,\infty)$.
    \end{enumerate}
\end{lemma}

\begin{remark}
    \label{rem:lemma-2}
    It follows from \eqref{eq:ODE-psi-h} that $\psi_{-}^h(x)$ satisfies
        \begin{align}
            \Gamma_{\ell}\psi_{-}^h(x)-\beta \psi_{-}^h(x)=-\theta(\mu_{h}-\mu_{\ell})(\psi_{-}^h)'(x), 
            \label{eq:ODE-psi-h-with-ell}
                \end{align}
    and it follows from \eqref{eq:ODE-psi-l} that $\psi_{+}^{\ell}(x)$ satisfies
        \begin{align}
            \Gamma_{h}\psi_{+}^{\ell}(x)-\beta \psi_{+}^{\ell}(x)=\theta(\mu_{h}-\mu_{\ell})(\psi_{+}^{\ell})'(x). 
            \label{eq:ODE-psi-ell-with-h}
        \end{align}
\end{remark}

Recall that $J_{h}(\cdot)$ is given by \eqref{eq:J-h}, $J_\ell(\cdot)$ is given by \eqref{eq:J-l}, and $J(\cdot)$ is given by \eqref{eq:J}. We will, in the following proposition,  provide the expected profit function under a given $(d,\infty)$-threshold policy, $(-\infty,u)$-threshold policy, and $(d,u)$-threshold policy, respectively.

\begin{proposition}
    \label{prop:threshold-policy}
    \begin{enumerate}[$(a)$]
        \item The expected profit function under a given $(d,\infty)$-threshold policy with $d\in(-\infty,\infty)$ is given by
        \begin{align}
            \mathcal{J}((x,\mu_\ell),(d,\infty))&=J_{\ell}(x),\label{eq:profit-threshold-1-1}\\
            \mathcal{J}((x,\mu_h),(d,\infty))&=
            \begin{cases}
                J_{\ell}(x)-K_{\ell} &\text{for $x\leq d$},\\
                J_h(x) + \frac{-J(d)-K_{\ell}}{\psi_{-}^h(d)}\psi_{-}^h(x) &\text{for $x>d$}.
            \end{cases}
            \label{eq:profit-threshold-1-2}
        \end{align}
        \item The expected profit function under a given $(-\infty,u)$-threshold policy with $u\in (-\infty,\infty)$ is given by
        \begin{align}   
            \mathcal{J}((x,\mu_h),(-\infty,u))&=J_{h}(x),\label{eq:profit-threshold-2-1}\\
            \mathcal{J}((x,\mu_\ell),(-\infty,u))&=
            \begin{cases}
                J_{h}(x)-K_h &\text{for $x\geq u$},\\
                J_\ell(x) + \frac{J(u)-K_h}{\psi_{+}^\ell(u)}\psi_{+}^\ell(x) &\text{for $x<u$}.
            \end{cases}
            \label{eq:profit-threshold-2-2}
        \end{align}
        \item The expected profit function under a given $(d,u)$-threshold policy with $d,u\in(-\infty,\infty)$ and $d<u$ is given by 
        \begin{align}
            \label{eq:profit-threshold-3-1}
            \mathcal{J}((x,\mu_h),(d,u))&=
            \begin{cases}
                J_\ell(x) + B(d,u)\psi_{+}^\ell(x)-K_{\ell} & \text{for $x\leq d$},\\
                J_h(x)+A(d,u)\psi_{-}^h(x)& \text{for $x>d$},
            \end{cases}\\
            \mathcal{J}((x,\mu_\ell),(d,u))&=
            \begin{cases}
                J_h(x)+A(d,u)\psi_{-}^h(x)-K_h &\text{for $x\geq u$},\\
                J_\ell(x) + B(d,u)\psi_{+}^\ell(x) &\text{for $x<u$},
            \end{cases}
            \label{eq:profit-threshold-3-2}
        \end{align}
        with
        \begin{align}
            A(d,u) &= \frac{-(J(d) + K_{\ell})\psi_+^{\ell}(u) + (J(u) - K_h)\psi_+^{\ell}(d)}{\psi_-^h(d)\psi_+^{\ell}(u) - \psi_+^{\ell}(d)\psi_-^h(u)}, \label{eq:A}\\
            B(d,u) &= \frac{-(J(d) + K_{\ell})\psi_-^h(u) + (J(u) - K_h)\psi_-^h(d)}{\psi_-^h(d)\psi_+^{\ell}(u) - \psi_+^{\ell}(d)\psi_-^h(u)}.\label{eq:B}
        \end{align}
        where $(A(d,u),B(d,u))$ is the unique solution of 
        \begin{align*}
            A \psi_-^h(d) - B \psi_+^{\ell}(d) +J(d) =- K_{\ell}\quad \text{and}\quad 
            A \psi_-^h(u) - B \psi_+^{\ell}(u) + J(u) =  K_h.
        \end{align*}
    \end{enumerate}
\end{proposition}

\begin{proof}[Proof of Proposition \ref{prop:threshold-policy}]
    $(a)$ Under a given $(d,\infty)$-threshold policy with $-\infty<d<\infty$, if the initial mode is $\mu_\ell$, the system will keep mode $\mu_{\ell}$ forever, i.e., 
\[
\mathcal{J}((x,\mu_{\ell}),(d,\infty))=J_{\ell}(x),
\]
which is \eqref{eq:profit-threshold-1-1}.
If the initial mode is $\mu_{h}$ and the initial state satisfies $x\leq d$, the system will immediately switch to mode $\mu_\ell$ and stay there, i.e., 
\begin{equation*}
    \mathcal{J}((x,\mu_h),(d,\infty))=J_{\ell}(x)-K_{\ell},
\end{equation*}
which is the first line in \eqref{eq:profit-threshold-1-2}. 

It remains to show the second line in \eqref{eq:profit-threshold-1-2}, i.e., the case when $x>d$ and the initial mode is $\mu_h$. 
Under the $(d,\infty)$-threshold policy, the system operates in mode $\mu_h$ until the process, starting from an initial state $x>d$, first hits $d$, and thereafter switches to mode $\mu_\ell$ permanently. 
Let $\tau_d=\inf\{t\geq0: X^{h}(t)\leq d\}$ be the first hitting time. By the strong Markov property, we have
\begin{align*}
    \mathcal{J}((x,\mu_h),(d,\infty))
    =& \E_{(x,\mu_h)}\left[
    \int_0^{\tau_d}e^{-\beta t}
    \left(r(X^{h}(t))^{+}-c\right)\dd t
    \right] \\
    & + \E_{(x,\mu_h)}\left[e^{-\beta\tau_d}\right]
    \left(J_\ell(d)-K_\ell\right)\quad \text{and}\\
    J_h(x)
    =& \E_{(x,\mu_h)}\left[
    \int_0^{\tau_d}e^{-\beta t}
    \left(r(X^{h}(t))^{+}-c\right)\dd t
    \right] 
    + \E_{(x,\mu_h)}\left[e^{-\beta\tau_d}\right]J_h(d),
\end{align*}
where $J_{h}$ and $J_{\ell}$ denote the expected profit functions under the no-switching policies with modes $\mu_h$ and $\mu_\ell$, respectively.
Hence
\begin{align}
\label{eq:V-d-2-app}
    \mathcal{J}((x,\mu_h),(d,\infty))= J_h(x)+\E_{(x,\mu_h)}\left[e^{-\beta \tau_d}\right]\left(-J(d) - K_{\ell}\right),
\end{align}
where $J(d)=J_h(d)-J_{\ell}(d)$ (see \eqref{eq:J}).
For the Ornstein--Uhlenbeck process $X^h$, it follows from Equation (2.6) in \cite{dayanik2003optimal} that the Laplace transform of the first hitting time $\tau_d$ from an initial state $x > d$ is given by 
\begin{align}
\label{eq:E-tau-app}
    \E_{x}\left[e^{-\beta \tau_{d}}\right] = \frac{\psi_{-}^h(x)}{\psi_{-}^h(d)},
\end{align}
where $\psi_{-}^h$ is defined in \eqref{eq:psi-h-}.
Combining \eqref{eq:V-d-2-app} with \eqref{eq:E-tau-app}, we obtain the second line in \eqref{eq:profit-threshold-1-2}.

($b$) Under a given $(-\infty,u)$-threshold policy with $-\infty<u<\infty$, if the initial mode is $\mu_h$, the system will keep mode $\mu_{h}$ forever, i.e., 
\begin{align*}
    \mathcal{J}((x,\mu_h),(-\infty,u))=J_{h}(x).
\end{align*}

If the initial mode is $\mu_{\ell}$ and the initial state satisfies $x\geq u$, the system will immediately switch to mode $\mu_h$ and stay there, i.e., 
\[
\mathcal{J}((x,\mu_\ell),(-\infty,u))=J_{h}(x)-K_h.
\]

If the initial mode is $\mu_{\ell}$ and the initial state satisfies $x<u$,  the system operates in mode $\mu_\ell$ until the process, starting from an initial state $x<u$, first hits $u$, and thereafter switches to mode $\mu_h$ permanently. 
Let $\tau_u=\inf\{t\geq0: X^{\ell}(t)\geq u\}$ be the first hitting time. By the strong Markov property, we have
\begin{align*}
    \mathcal{J}((x,\mu_\ell),(-\infty,u))
    =& \E_{(x,\mu_\ell)}\left[ \int_0^{\tau_{u}}e^{-\beta t}\left( r(X^{\ell}(t))^{+} \right)\dd t \right] \\
    & + \E_{(x,\mu_\ell)}\left[e^{-\beta \tau_u}\right]\left(J_h(u) - K_h\right) \quad \text{and}\\[4pt]
    J_\ell(x)
    = &\E_{(x,\mu_\ell)}\left[ \int_0^{\tau_{u}}e^{-\beta t}\left( r(X^{\ell}(t))^{+} \right)\dd t \right] 
     + \E_{(x,\mu_\ell)}\left[e^{-\beta \tau_u}\right] J_\ell(u).
\end{align*}
Then $\mathcal{J}((x,\mu_\ell),(-\infty,u))= J_\ell(x)+\E_{(x,\mu_\ell)}\left[e^{-\beta \tau_u}\right] \left(J(u)-K_h\right)$,
which, together with $\E_{(x,\mu_\ell)}\left[e^{-\beta \tau_u}\right]=\psi_{+}^\ell(x)/\psi_{+}^\ell(u)$ (see also \cite{dayanik2003optimal}), implies the second line in \eqref{eq:profit-threshold-2-2}.

($c$) Recall that under a given $(d, u)$-threshold policy with $-\infty < d < u < \infty$, the system switches from mode $\mu_h$ to $\mu_{\ell}$ when the state falls to $d$, and switches from $\mu_{\ell}$ to $\mu_h$ when the state rises to $u$.
For fixed finite thresholds $d<u$, the strong Markov property and the discounted first-passage identities imply that the expected discounted switching cost is bounded by a convergent geometric series. Together with the integrability of the discounted running payoff, this ensures that the policy profits below are finite.

We first consider the case where the initial mode is $\mu_h$. If the initial state satisfies $x \le d$, the system immediately switches to mode $\mu_{\ell}$. In this case, the expected profit is:
\begin{equation}
\label{eq:prop22c_h_immediate}
    \mathcal{J}((x, \mu_h), (d, u)) = \mathcal{J}((x, \mu_{\ell}), (d, u)) - K_{\ell}, \quad \text{for } x \le d.
\end{equation}

If the initial state satisfies $x > d$, the system operates in mode $\mu_h$ until the process first hits $d$. Let $\tau_d = \inf\{t \ge 0 : X^{h}(t) \le d\}$ be this first hitting time. By the strong Markov property, the expected discounted net profit is
\begin{align*}
    \mathcal{J}((x, \mu_h), (d, u)) &= \mathbb{E}_{(x,\mu_h)} \left[ \int_0^{\tau_d} e^{-\beta t} (r(X^{h}(t))^+ - c) dt \right] \\
    &\quad+ \mathbb{E}_{(x,\mu_h)} \left[ e^{-\beta \tau_d} \right] \left( \mathcal{J}((d, \mu_{\ell}), (d, u)) - K_{\ell} \right).
\end{align*}
Similarly, the expected profit $J_h(x)$ under the no-switching policy  can be decomposed as
\begin{equation*}
    J_h(x) = \mathbb{E}_{(x,\mu_h)} \left[ \int_0^{\tau_d} e^{-\beta t} (r(X^{h}(t))^+ - c) dt \right] + \mathbb{E}_{(x,\mu_h)} \left[ e^{-\beta \tau_d} \right] J_h(d).
\end{equation*}
Subtracting the second equation from the first yields
\begin{equation*}
    \mathcal{J}((x, \mu_h), (d, u)) - J_h(x) = \mathbb{E}_{(x,\mu_h)} \left[ e^{-\beta \tau_d} \right] \left( \mathcal{J}((d, \mu_{\ell}), (d, u))  - J_h(d)- K_{\ell} \right).
\end{equation*}
Substituting \eqref{eq:E-tau-app} into the equation above, we obtain
\begin{equation}
\label{eq:prop22c_h_cont}
    \mathcal{J}((x, \mu_h), (d, u)) = J_h(x) + \frac{\mathcal{J}((d, \mu_{\ell}), (d, u)) - J_h(d)- K_{\ell}}{\psi_-^h(d)} \psi_-^h(x), \quad \text{for } x > d.
\end{equation}
Combining \eqref{eq:prop22c_h_immediate} and \eqref{eq:prop22c_h_cont} yields 
\begin{align}
\label{eq:prop22V-1}
    \mathcal{J}((x, \mu_h), (d, u)) = 
    \begin{cases}
        \mathcal{J}((x, \mu_{\ell}), (d, u))-K_{\ell} & \text{for $x\leq d$},\\
        J_h(x) + \frac{\mathcal{J}((d, \mu_{\ell}), (d, u)) - J_h(d)- K_{\ell}}{\psi_-^h(d)} \psi_-^h(x)& \text{for $x>d$}.
    \end{cases}
\end{align}
    
By a completely symmetric argument, if the initial mode is $\mu_{\ell}$, an immediate switch occurs if $x \ge u$, yielding:
\begin{equation}\label{eq:prop22c_l_immediate}
    \mathcal{J}((x, \mu_{\ell}), (d, u)) = \mathcal{J}((x, \mu_h), (d, u)) - K_h, \quad \text{for } x \ge u.
\end{equation}
For $x < u$, the system operates in $\mu_{\ell}$ until the first hitting time $\tau_u = \inf\{t \ge 0 : X^{\ell}(t) \ge u\}$. Using the strong Markov property and the Laplace transform $\mathbb{E}_{(x,\mu_{\ell})}[e^{-\beta \tau_u}] = \psi_+^\ell(x)/\psi_+^\ell(u)$, we have:
\begin{equation}\label{eq:prop22c_l_cont}
    \mathcal{J}((x, \mu_{\ell}), (d, u)) = J_\ell(x) + \frac{\mathcal{J}((u, \mu_h), (d, u)) - J_\ell (u) - K_h}{\psi_+^\ell(u)} \psi_+^\ell(x), \quad \text{for } x < u.
\end{equation}
Equations \eqref{eq:prop22c_l_immediate} and \eqref{eq:prop22c_l_cont} establish
\begin{align}
\label{eq:prop22V-2}
     \mathcal{J}((x, \mu_{\ell}), (d, u)) =
     \begin{cases}
         \mathcal{J}((x, \mu_h), (d, u)) - K_h, & \text{for } x \ge u,\\
         J_\ell(x) + \frac{\mathcal{J}((u, \mu_h), (d, u)) - J_\ell (u) - K_h}{\psi_+^\ell(u)} \psi_+^\ell(x), & \text{for } x < u.
     \end{cases}
\end{align}

To obtain a fully explicit characterization of the expected discounted profit functions in the continuation regions, we define the constants:
\begin{equation*}
    A = \frac{\mathcal{J}((d, \mu_{\ell}), (d, u)) - K_{\ell} - J_h(d)}{\psi_-^h(d)} \quad \text{and} \quad B = \frac{\mathcal{J}((u, \mu_h), (d, u)) - J_\ell(u) - K_h}{\psi_+^{\ell}(u)}.
\end{equation*}
From \eqref{eq:prop22c_l_cont}, evaluated at $x=d$, we know $\mathcal{J}((d, \mu_{\ell}), (d, u)) = J_{\ell}(d) + B \psi_+^{\ell}(d)$. Substituting this into the definition of $A$ gives:
\begin{equation} 
\label{eq:A-B-equation-proof-1}
    A \psi_-^h(d) = J_{\ell}(d) + B \psi_+^{\ell}(d) - J_h(d)- K_{\ell}  \implies A \psi_-^h(d) - B \psi_+^{\ell}(d) = -J(d) - K_{\ell}.
\end{equation}
Similarly, evaluating \eqref{eq:prop22c_h_cont} at $x=u$ gives $\mathcal{J}((u, \mu_h), (d, u)) = J_h(u) + A \psi_-^h(u)$. Substituting this into the definition of $B$ yields:
\begin{equation} 
\label{eq:A-B-equation-proof-2}
    B \psi_+^{\ell}(u) = J_h(u) + A \psi_-^h(u) - J_{\ell}(u) - K_h \implies -A \psi_-^h(u) + B \psi_+^{\ell}(u) = J(u) - K_h.
\end{equation}
Thus, equations \eqref{eq:A-B-equation-proof-1} and \eqref{eq:A-B-equation-proof-2} form a non-singular $2 \times 2$ linear system for $A$ and $B$. Let $\Delta(d, u) = \psi_-^h(d)\psi_+^{\ell}(u) - \psi_+^{\ell}(d)\psi_-^h(u)$. Since $d < u$, and considering the strict monotonicities of $\psi_-^h$ (decreasing) and $\psi_+^{\ell}$ (increasing); see Lemma \ref{lem:psi} ($b$), we have $\Delta(d, u) > 0$. By Cramer's rule, the solution to \eqref{eq:A-B-equation-proof-1} and \eqref{eq:A-B-equation-proof-2} is
\begin{align*}
    A &= \frac{-(J(d) + K_{\ell})\psi_+^{\ell}(u) + (J(u) - K_h)\psi_+^{\ell}(d)}{\Delta(d, u)}, \\
    B &= \frac{-(J(d) + K_{\ell})\psi_-^h(u) + (J(u) - K_h)\psi_-^h(d)}{\Delta(d, u)}.
\end{align*}
Substituting $A$ and $B$ back into \eqref{eq:prop22V-1} and \eqref{eq:prop22V-2} completely determines the expected profit functions under the $(d, u)$-threshold policy.
\end{proof}

\section{Main results}

\label{sec:mainresults}

In this section, we will present our main results in Theorem \ref{thm:main}, which shows the optimal policies under different conditions of system parameters. 

Let $V(x,\mu)$ denote the optimal profit function for initial state $x\in\mathbb{R}$ and initial mode $\mu\in\mathcal{U}=\{\mu_\ell,\mu_h\}$. The optimal policy consists of five cases, determined by the relative ordering of the four benchmark levels $-\frac{c}{\beta}$, $\frac{r\theta(\mu_{h}-\mu_{\ell})}{\beta(\beta+\theta)}-\frac{c}{\beta}$, $-K_{\ell}$, and $K_h$, where the inequalities $-\frac{c}{\beta}<\frac{r\theta(\mu_{h}-\mu_{\ell})}{\beta(\beta+\theta)}-\frac{c}{\beta}$ and $-K_{\ell}<K_h$ hold automatically. 
Figure~\ref{fig:optimal-policy-classification} illustrates these cases for fixed $K_{\ell}$ and $K_h$ as the other system parameters vary.

\begin{figure}[t]
      \centering
      \includegraphics[width=0.8\linewidth]{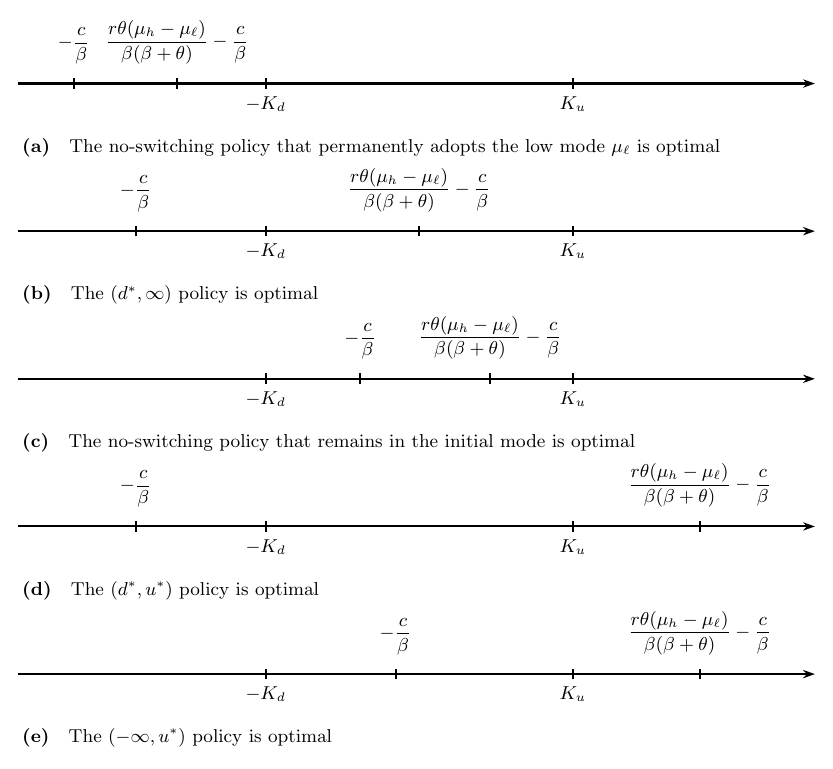}
      \caption{Classification of optimal policies across parameter regimes}
      \label{fig:optimal-policy-classification}
\end{figure}

Recall the expressions of $J_h$ in \eqref{eq:J-h}, $J_{\ell}$ in \eqref{eq:J-l}, $J$ in \eqref{eq:J}, $\psi_-^h$ in \eqref{eq:psi-h-}, and $\psi_+^\ell$ in \eqref{eq:psi-l+}.

\begin{theorem}
\label{thm:main}
    \begin{enumerate}[$(a)$]
        \item If $\frac{r\theta(\mu_{h}-\mu_{\ell})}{\beta(\beta+\theta)}-\frac{c}{\beta}\leq -K_{\ell}$, then the no-switching policy that permanently adopts the low mode $\mu_\ell$ strictly after time zero is optimal. The optimal profit functions are given by
        $V(x,\mu_h)=J_\ell(x)-K_{\ell}$ and $V(x,\mu_\ell)=J_\ell(x)$.
        \item If $-\frac{c}{\beta}<-K_{\ell} < \frac{r\theta(\mu_{h}-\mu_{\ell})}{\beta(\beta+\theta)}-\frac{c}{\beta}\leq K_h$, then the $(d^*,\infty)$ policy is optimal, with the optimal parameter $d^*$ uniquely determined by smooth-fit condition
        \begin{align}
        \label{eq:smooth-condition-1}
            J'(d^*) + \frac{-J(d^*) - K_{\ell}}{\psi_-^h(d^*)}(\psi_-^h)'(d^*)=0.
        \end{align}
        Furthermore, the optimal profit functions are given by 
        \begin{align*}
        V(x,\mu_h)=\mathcal{J}((x,\mu_h),(d^*,\infty))\quad \text{and}\quad V(x,\mu_\ell)=\mathcal{J}((x,\mu_\ell),(d^*,\infty)),
        \end{align*}
        as defined in \eqref{eq:profit-threshold-1-1}-\eqref{eq:profit-threshold-1-2}.        
        \item If $-K_{\ell}\leq -\frac{c}{\beta} < \frac{r\theta(\mu_{h}-\mu_{\ell})}{\beta(\beta+\theta)}-\frac{c}{\beta}\leq K_h$, then the no-switching policy that retains the initial mode forever is optimal, and optimal profit functions are given by $V(x,\mu_h)=J_h(x)$ and $V(x,\mu_\ell)=J_\ell(x)$.
        \item If $-\frac{c}{\beta}< -K_{\ell} <K_h< \frac{r\theta(\mu_{h}-\mu_{\ell})}{\beta(\beta+\theta)}-\frac{c}{\beta}$, then the $(d^*,u^*)$ policy with $-\infty<d^*<u^*<\infty$ is optimal, with the optimal parameters $d^*$ and $u^*$ uniquely determined by smooth-fit conditions 
        \begin{align}
            J'(d^*)& =-A(d^*,u^*)(\psi_-^h)'(d^*)+B(d^*,u^*)(\psi_{+}^\ell)'(d^*), \label{eq:smooth-condition-3}\\
            J'(u^*)& =-A(d^*,u^*)(\psi_-^h)'(u^*)+B(d^*,u^*)(\psi_{+}^\ell)'(u^*), \label{eq:smooth-condition-4}
        \end{align}
        where $A(d,u)$ and $B(d,u)$ are given by \eqref{eq:A}-\eqref{eq:B}. Furthermore, the optimal profit functions are given by 
        \begin{align*}
            V(x,\mu_h)=\mathcal{J}((x,\mu_h),(d^*,u^*))\quad \text{and}\quad V(x,\mu_\ell)=\mathcal{J}((x,\mu_\ell),(d^*,u^*)), 
        \end{align*}
        as defined in \eqref{eq:profit-threshold-3-1}-\eqref{eq:profit-threshold-3-2} with $d=d^*$ and $u=u^*$.
        \item If $-K_{\ell}\leq-\frac{c}{\beta} <K_h< \frac{r\theta(\mu_{h}-\mu_{\ell})}{\beta(\beta+\theta)}-\frac{c}{\beta}$, then the $(-\infty,u^*)$ policy is optimal, with the optimal parameter $u^*$ uniquely determined by smooth-fit condition
        \begin{align}
        \label{eq:smooth-condition-2}
            J'(u^*) - \frac{J(u^*) - K_h}{\psi_+^\ell(u^*)}(\psi_+^\ell)'(u^*)=0.
        \end{align}
        Furthermore, the optimal profit functions are given by 
        \begin{align*}
        V(x,\mu_h)=\mathcal{J}((x,\mu_h),(-\infty,u^*))\quad \text{and}\quad V(x,\mu_\ell)=\mathcal{J}((x,\mu_\ell),(-\infty,u^*)),
        \end{align*}
        as defined in \eqref{eq:profit-threshold-2-1}-\eqref{eq:profit-threshold-2-2} with $u=u^*$.
    \end{enumerate}
\end{theorem}
 
The smooth-fit condition \eqref{eq:smooth-condition-1} is obtained by the continuity of $\mathcal{J}_x((x,\mu_h),(d^*,\infty))$ at $d=d^*$, i.e., $\lim_{x\uparrow d^*}\mathcal{J}_x((x,\mu_h),(d^*,\infty))=\lim_{x\downarrow d^*}\mathcal{J}_x((x,\mu_h),(d^*,\infty))$, where we let $\mathcal{J}_x$ denote $\partial \mathcal{J}/\partial x$. Similarly, the smooth-fit conditions \eqref{eq:smooth-condition-3}-\eqref{eq:smooth-condition-4} are obtained by the continuity of $\mathcal{J}_x((x,\mu_h),(d^*,u^*))$ and $\mathcal{J}_x((x,\mu_\ell),(d^*,u^*))$ at $d^*$ and $u^*$, and the smooth-fit condition \eqref{eq:smooth-condition-2} is obtained by the continuity of $\mathcal{J}_x((x,\mu_\ell),(-\infty,u^*))$ at $u=u^*$.

We now provide an economic interpretation of Theorem~\ref{thm:main} using a single-switch benchmark. Recall that $J(x)=J_h(x)-J_\ell(x)$ represents the discounted profit difference between permanently operating in the high and low modes. If switching is allowed only once at the initial time, then switching from $\mu_h$ to $\mu_\ell$ yields a net gain of $-J(x)-K_\ell$, while switching from $\mu_\ell$ to $\mu_h$ yields a net gain of $J(x)-K_h$. These two quantities provide a benchmark for understanding which switching directions can potentially be profitable.

In case~($a$), the range of $J(x)$ lies below $-K_\ell$, so a downward switch is always profitable under this benchmark. The high mode cannot justify its additional operating cost, and the low mode is preferred permanently.

In case~($b$), $J(x)$ crosses $-K_\ell$ but remains below $K_h$. Hence, only downward switching can generate a potential gain, while upward switching is never attractive. This leads to a policy with downward switching only.

In case~($c$), the range of $J(x)$ is contained between $-K_\ell$ and $K_h$. Neither switching direction can recover its cost, and retaining the initial mode is optimal.

In case~($d$), both $K_{h}$ and $K_{\ell}$ lie within the range of $J(x)$. Downward switching becomes attractive at low states and upward switching becomes attractive at high states, leading to a two-sided switching policy with an inaction region.

In case~($e$), the range of $J(x)$ lies above $-K_\ell$ and extends beyond $K_h$. Only upward switching can be profitable, so the high mode is preferred at sufficiently large states.

The optimal switching thresholds should be distinguished from the states at which the permanent-mode profit differences equal the switching costs. Since immediate switching is available, the optimal profit functions satisfy
\[
-K_{\ell}
\leq V(x,\mu_h)-V(x,\mu_\ell)
\leq K_h.
\]
At an optimal downward switching boundary, the lower equality holds; at an optimal upward switching boundary, the upper equality holds. These comparisons involve optimal future profits, including subsequent switching opportunities. In particular, the smooth-fit condition in case~($b$) implies $J(d^*)<-K_{\ell}$, while that in case~($e$) implies $J(u^*)>K_h$. These strict inequalities illustrate the value of waiting: even when an immediate permanent change of mode has a positive net benefit, postponing that change may still be optimal. 

We next prove Theorem \ref{thm:main} in the following sections by two steps: First, in \S\ref{sec:parameters}, we will show the existence and uniqueness of the finite policy parameters proposed in Theorem \ref{thm:main} ($b$), ($d$) and ($e$), and then in \S\ref{sec:optimality}, we will prove the optimality of proposed policies in all five cases of Theorem \ref{thm:main}. 

\section{Existence and Uniqueness of Optimal Policy Parameters}
\label{sec:parameters}

In this section, we establish the existence and uniqueness of the policy parameters appearing in Theorem~\ref{thm:main}($b$), ($d$), and ($e$). To treat these three cases within a unified framework, define 
\begin{align}
F_{A,B}(x)&=A\psi_-^h(x)-B\psi_+^\ell(x)+J(x),\label{eq:F}\\
f_{A,B}(x)&=F_{A,B}'(x)
=A(\psi_-^h)'(x)-B(\psi_+^\ell)'(x)+J'(x),\label{eq:f}
\end{align}
for $x,A,B\in\mathbb{R}$.

For the single lower threshold in Theorem~\ref{thm:main}($b$), consider
\begin{equation}
\label{eq:optimal-equation-d}
F_{A,0}(d)=-K_{\ell},
\qquad
f_{A,0}(d)=0.
\end{equation}
Since $\psi_-^h(d)>0$, the first equation uniquely determines
\begin{equation}
\label{eq:A-2}
A=\frac{-J(d)-K_{\ell}}{\psi_-^h(d)}.
\end{equation}
Substitution into the second equation gives
\[
J'(d)+\frac{-J(d)-K_{\ell}}{\psi_-^h(d)}(\psi_-^h)'(d)=0,
\]
which is exactly \eqref{eq:smooth-condition-1}. 
Conversely, any solution $d$ of \eqref{eq:smooth-condition-1}, together with the coefficient $A$ defined above, satisfies
\eqref{eq:optimal-equation-d}. Thus the two formulations are equivalent.

For the two finite thresholds in Theorem~\ref{thm:main}($d$), consider the system
\begin{align}
F_{A,B}(d)&=-K_{\ell},\qquad F_{A,B}(u)=K_h,
\label{eq:optimal-equation-du-1}\\
f_{A,B}(d)&=0,\qquad f_{A,B}(u)=0.
\label{eq:optimal-equation-du-2}
\end{align}
The equations in \eqref{eq:optimal-equation-du-1} are called \textit{value-matching equations}, and those in \eqref{eq:optimal-equation-du-2} are called \textit{smooth-fit equations}.
For any fixed finite $d<u$, positivity and strict monotonicity of $\psi_-^h$ and $\psi_+^\ell$ (see Lemma \ref{lem:psi}) imply
\[
\psi_-^h(d)\psi_+^\ell(u)-\psi_-^h(u)\psi_+^\ell(d)>0.
\]
Thus the two linear equations in \eqref{eq:optimal-equation-du-1} uniquely determine $A=A(d,u)$ and $B=B(d,u)$ as given in \eqref{eq:A} and \eqref{eq:B}. Substituting these coefficients into \eqref{eq:optimal-equation-du-2} gives \eqref{eq:smooth-condition-3} and \eqref{eq:smooth-condition-4}. Conversely, any finite pair $d<u$ satisfying these smooth-fit conditions, together with $A(d,u)$ and $B(d,u)$, solves \eqref{eq:optimal-equation-du-1}--\eqref{eq:optimal-equation-du-2}. 

Finally, for the single upper threshold in Theorem~\ref{thm:main}($e$), consider 
\begin{equation}
\label{eq:optimal-equation-u}
F_{0,B}(u)=K_h,
\qquad
f_{0,B}(u)=0.
\end{equation}
The first equation uniquely determines
\[
B=\frac{J(u)-K_h}{\psi_+^\ell(u)}.
\]
Substitution into the second equation gives
\[
J'(u)-\frac{J(u)-K_h}{\psi_+^\ell(u)}(\psi_+^\ell)'(u)=0.
\]
Hence \eqref{eq:optimal-equation-u} is equivalent to \eqref{eq:smooth-condition-2}.

The following proposition establishes existence and uniqueness for all three systems.

\begin{proposition}
\label{prop:existence}
\begin{enumerate}[$(a)$]
\item
If $-\frac{c}{\beta}<-K_{\ell}<\frac{r\theta(\mu_h-\mu_\ell)}{\beta(\beta+\theta)}-\frac{c}{\beta}\leq K_h$, then the system in \eqref{eq:optimal-equation-d} has a unique finite solution $(d^*,A^*)$, with $A^*>0$.
\item
If $-\frac{c}{\beta}<-K_{\ell}<K_h<\frac{r\theta(\mu_h-\mu_\ell)}{\beta(\beta+\theta)}-\frac{c}{\beta}$, then the system
in \eqref{eq:optimal-equation-du-1}-\eqref{eq:optimal-equation-du-2} has a unique finite solution $(d^*,u^*,A^*,B^*)$, with $d^*<u^*$, $A^*>0$, and $B^*>0$.
\item If $-K_{\ell}\leq -\frac{c}{\beta}<K_h<\frac{r\theta(\mu_h-\mu_\ell)}{\beta(\beta+\theta)}-\frac{c}{\beta}$, then the system in \eqref{eq:optimal-equation-u} has a unique finite solution $(u^*,B^*)$, with $B^*>0$.
\end{enumerate}
\end{proposition}

The proof of Proposition~\ref{prop:existence} proceeds through a sequence of auxiliary results. 

First, Lemma~\ref{lem:minimizer} claims that the local minima (if it exists) of $f_{A,B}$ must be strictly positive whenever at least one coefficient is nonnegative. This property is used in Lemma~\ref{lem:A>0,B>0} to establish positivity of the coefficients in any finite solution of the two-threshold system with $d<u$, and in Lemma~\ref{lem:fAB} to characterize the interval on which $f_{A,B}$ is positive with two boundary points, denoted by $d(A,B)$ and $u(A,B)$. 

We then introduce the area function $S(A,B)=\int_{\mathbb{R}}(f_{A,B}(x))^+\,\mathrm{d}x$. Its continuity and strict comparison property, established in Lemma~\ref{lem:S-properties}, allow us to describe each admissible positive level set, i.e., $\{(A,B): S(A,B)=\kappa\}$ for each $\kappa>0$,  as a continuous, strictly decreasing curve $A=A_\kappa(B)$ in Lemma~\ref{lem:area-level}. For prescribed admissible matching values $\kappa_1<\kappa_2$, we first impose $S(A,B)=\kappa:=\kappa_2-\kappa_1$ and then select the unique point on this curve at which $F_{A_{\kappa}(B),B}(d(A_{\kappa}(B),B))=\kappa_1$. The second matching condition follows from the area identity. Lemma~\ref{lem:level-matching} proves that this selection is well defined and unique by showing that the lower matching value varies continuously and strictly monotonically along the curve.

The three parts of the proposition follow by choosing the appropriate matching values: choosing $\kappa_1=-K_{\ell}$ and $\kappa_2=\frac{r\theta(\mu_h-\mu_\ell)}{\beta(\beta+\theta)}-\frac{c}{\beta}$ for part ($a$), $\kappa_1=-K_{\ell}$ and $\kappa_2=K_h$ for part ($b$), $\kappa_1=-\frac{c}{\beta}$ and $\kappa_2=K_h$ for part ($c$). 
Finally, we show that every finite solution of each system satisfies the conditions of Lemma \ref{lem:level-matching}, which establishes uniqueness among all finite solutions.

The proofs of Lemmas \ref{lem:minimizer}-\ref{lem:fAB} and Lemmas \ref{lem:S-properties}-\ref{lem:level-matching} are given in Appendix \ref{app:section-4}.
\begin{lemma}
\label{lem:minimizer}
Let $A,B\in\mathbb{R}$, and suppose that at least one of $A$ and $B$ is nonnegative. Then every local minimizer $x_{\min}\in\mathbb{R}$ of $f_{A,B}$, if it exists, satisfies $f_{A,B}(x_{\min})>0$.
\end{lemma}

\begin{lemma}
\label{lem:A>0,B>0}
Every solution $(d,u,A,B)\in\mathbb{R}^4$ to \eqref{eq:optimal-equation-du-1}--\eqref{eq:optimal-equation-du-2} with $d<u$ satisfies $A>0$ and $B>0$.
\end{lemma}

\begin{lemma}
\label{lem:fAB}
Let $A,B\geq0$ with $A+B>0$, and suppose that $\sup_{x\in\mathbb{R}}f_{A,B}(x)>0$. Then there exist unique extended real numbers
$d(A,B)$ and $u(A,B)$, with
\[
-\infty\leq d(A,B)<u(A,B)\leq+\infty,
\]
such that
\begin{equation}
\label{eq:fAB-sign}
f_{A,B}(x)
\begin{cases}
<0, & x<d(A,B),\\
>0, & d(A,B)<x<u(A,B),\\
<0, & x>u(A,B),
\end{cases}
\end{equation}
where $d(A,B)=-\infty$ if and only if $A=0$ and  $u(A,B)=\infty$ if and only if $B=0$. Moreover, we have
\begin{align}
f_{A,B}\bigl(d(A,B)\bigr)&=0,
\qquad
f_{A,B}'\bigl(d(A,B)\bigr)>0
&&\text{if }A>0,\label{eq:fAB-lower-slope}\\
f_{A,B}\bigl(u(A,B)\bigr)&=0,
\qquad
f_{A,B}'\bigl(u(A,B)\bigr)<0
&&\text{if }B>0.\label{eq:fAB-upper-slope}
\end{align}

\end{lemma}

\begin{remark}
\label{rem:nonempty-set}
The parameter set considered in Lemma~\ref{lem:fAB} is nonempty;
in fact, it contains pairs with both coefficients strictly positive.
To see this, fix any $x_0\in\mathbb{R}$.
By Lemmas~\ref{lem:J}$(d)$ and~\ref{lem:psi}$(b)$, we have $J'(x_0)>0$, $(\psi_-^h)'(x_0)<0$, and $(\psi_+^\ell)'(x_0)>0$.
Define
\[
A_0=-\frac{J'(x_0)}{2(\psi_-^h)'(x_0)}>0,
\qquad
B_0=\frac{J'(x_0)}{4(\psi_+^\ell)'(x_0)}>0.
\]
Then
\[
f_{A_0,B_0}(x_0)
=-\frac{J'(x_0)}{2}-\frac{J'(x_0)}{4}+J'(x_0)
=\frac{J'(x_0)}{4}>0.
\]
Consequently, $\sup_{x\in\mathbb{R}}f_{A_0,B_0}(x)>0$.
\end{remark}

For $A,B\geq0$, define
\[
S(A,B):=\int_{\mathbb{R}}
\bigl(f_{A,B}(x)\bigr)^+\,\mathrm{d}x.
\]
If $A+B>0$ and $\sup_{x\in\mathbb{R}}f_{A,B}(x)>0$, then Lemma \ref{lem:fAB} gives
\[
S(A,B)
=
\int_{d(A,B)}^{u(A,B)}
f_{A,B}(x)\,\mathrm{d}x
=
F_{A,B}\bigl(u(A,B)\bigr)
-
F_{A,B}\bigl(d(A,B)\bigr),
\]
where values at infinite arguments are interpreted as limits. If $\sup_{x\in\mathbb{R}}f_{A,B}(x)\leq0$, then $S(A,B)=0$.
We also note that $\sup_{x\in\mathbb{R}}f_{A,B}(x)>0$ is equivalent to $S(A,B)>0$.

\begin{lemma}
\label{lem:S-properties}
The function $S$ is finite and continuous on $[0,\infty)^2$.
Moreover,
\[
S(0,0)=\frac{r\theta(\mu_h-\mu_\ell)}{\beta(\beta+\theta)},\qquad
\lim_{A\to\infty}S(A,0)=\lim_{B\to\infty}S(0,B)=0.
\]
It is nonincreasing in each variable. More precisely, if $A_2\geq A_1\geq0$, $B_2\geq B_1\geq0$, $(A_2,B_2)\neq(A_1,B_1)$, and $S(A_1,B_1)>0$, then
\begin{equation}
    \label{eq:S-monotone}
    S(A_2,B_2)<S(A_1,B_1).
\end{equation}
\end{lemma}

\begin{lemma}
\label{lem:area-level}
For every $\kappa\in\bigl(0,\frac{r\theta(\mu_h-\mu_\ell)}{\beta(\beta+\theta)}\bigr)$, there exist unique constants $\bar A_\kappa,\bar B_\kappa>0$
satisfying
\begin{align*}
    S(\bar{A}_\kappa,0)=S(0,\bar{B}_\kappa)=\kappa.
\end{align*}
For each $B\in[0,\bar B_\kappa]$, there exists a unique $A_\kappa(B)\geq0$ such that
\begin{align*}
    S\bigl(A_\kappa(B),B\bigr)=\kappa.
\end{align*}
The function $A_\kappa(\cdot)\in\mathcal{C}^1((0,\bar B_\kappa))$ is continuous and strictly decreasing on $[0,\bar B_\kappa]$, and satisfies
\begin{align}
\label{eq:A-kappa-bound}
A_\kappa(0)=\bar{A}_\kappa,
\qquad
A_\kappa(\bar B_\kappa)=0,
\qquad
0<A_\kappa(B)<\bar A_\kappa
\quad\text{for }0<B<\bar{B}_\kappa.
\end{align}
Moreover, every nonnegative solution of $S(A,B)=\kappa$ is of the form $(A_\kappa(B),B)$ with $B\in[0,\bar{B}_\kappa]$.
\end{lemma}

\begin{lemma}
\label{lem:level-matching}
Let $\kappa_1,\kappa_2\in\mathbb{R}$ satisfy $-\frac{c}{\beta}\leq \kappa_1<\kappa_2\leq \frac{r\theta(\mu_h-\mu_\ell)}{\beta(\beta+\theta)}-\frac{c}{\beta}$ and $(\kappa_1,\kappa_2)\neq (-\frac{c}{\beta}, \frac{r\theta(\mu_h-\mu_\ell)}{\beta(\beta+\theta)}-\frac{c}{\beta})$. 
Then there exists a unique pair $A,B\geq0$ with $A+B>0$ and $S(A,B)>0$ such that
\begin{equation}
\label{eq:F-kappa1-kappa2}
F_{A,B}\bigl(d(A,B)\bigr)=\kappa_1,\qquad F_{A,B}\bigl(u(A,B)\bigr)=\kappa_2,
\end{equation}
where values at infinite arguments are interpreted
as limits.
Moreover,
\[
A=0\iff \kappa_1=-\frac{c}{\beta},
\qquad
B=0\iff \kappa_2=\frac{r\theta(\mu_h-\mu_\ell)}{\beta(\beta+\theta)}-\frac{c}{\beta}.
\]
Consequently, $d(A,B)$ is finite if and only if $\kappa_1>-\frac{c}{\beta}$, and $u(A,B)$ is finite if and only if $\kappa_2<\frac{r\theta(\mu_h-\mu_\ell)}{\beta(\beta+\theta)}-\frac{c}{\beta}$.
\end{lemma}


\begin{figure}
    \centering
    \includegraphics[width=0.6\linewidth]{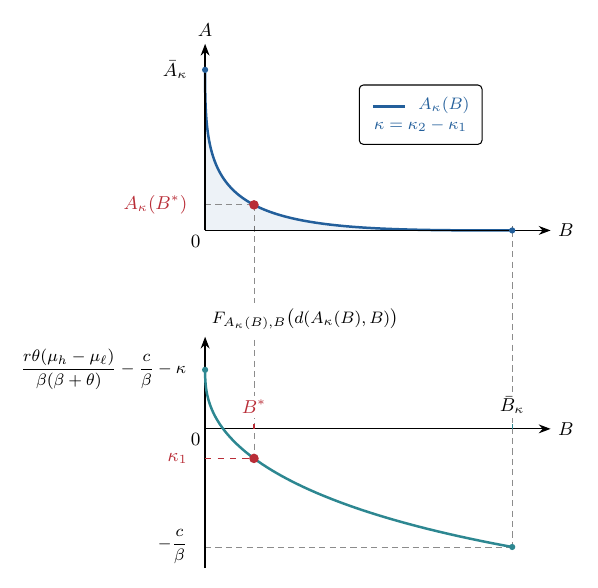}
    \caption{The figure above illustrates the curve $A_\kappa(B)$ in Lemma \ref{lem:area-level}, while the two figures together illustrates the pair $(A,B)$ in Lemma \ref{lem:level-matching} with $B=B^*$ and $A=A_\kappa(B^*)$.}
    \label{fig:lemma-4-6}
\end{figure}

Figure \ref{fig:lemma-4-6} illustrates the curve $A_\kappa(B)$ in Lemma \ref{lem:area-level} by $S(A,B)=\kappa$ and the pair $(A,B)$ in Lemma \ref{lem:level-matching} by characterizing $B=B^*$ using $F_{(A_k(B),B)}(d(A_k(B),B))=\kappa_1$ and letting $A=A_\kappa(B^*)$. 
Now we are ready to prove Proposition \ref{prop:existence}.

\begin{proof}[Proof of Proposition \ref{prop:existence}]
($a$) Apply Lemma \ref{lem:level-matching} by letting $\kappa_1=-K_{\ell}$ and $\kappa_2= \frac{r\theta(\mu_h-\mu_\ell)}{\beta(\beta+\theta)}-\frac{c}{\beta}$. It yields a pair $(A^*,0)$ with $A^*>0$ and a finite zero $d^*=d(A^*,0)$ satisfying
\[
F_{A^*,0}(d^*)=-K_{\ell},
\qquad
f_{A^*,0}(d^*)=0.
\]
This proves existence.

Let $(d,A)$ be any finite solution of this system. The equation $f_{A,0}(d)=0$ implies
\[
A=-\frac{J'(d)}{(\psi_-^h)'(d)}>0.
\]
Since $\psi_-^h(x)\to0$ as $x\to\infty$,
\[
\lim_{x\to\infty}F_{A,0}(x)=\lim_{x\to\infty}\left[A\psi_-^h(x)+J(x)\right]=\frac{r\theta(\mu_h-\mu_\ell)}{\beta(\beta+\theta)}-\frac{c}{\beta},
\]
where the last equality follows from \eqref{eq:limit-psi-h-2} and \eqref{eq:lim-J}. Consequently,
\[
\int_d^\infty f_{A,0}(x)\,\mathrm{d}x
=\frac{r\theta(\mu_h-\mu_\ell)}{\beta(\beta+\theta)}-\frac{c}{\beta}-F_{A,0}(d)
=\frac{r\theta(\mu_h-\mu_\ell)}{\beta(\beta+\theta)}-\frac{c}{\beta}+K_{\ell}
>0.
\]
Hence $\sup_{x\in\mathbb{R}}f_{A,0}(x)>0$. Lemma \ref{lem:fAB} shows that $d$ is its unique zero, so $d=d(A,0)$ and $u(A,0)=\infty$. Thus $(A,0)$ satisfies the conditions in \eqref{eq:F-kappa1-kappa2} of Lemma \ref{lem:level-matching} with $(\kappa_1,\kappa_2)=(-K_{\ell},\frac{r\theta(\mu_h-\mu_\ell)}{\beta(\beta+\theta)}-\frac{c}{\beta})$. Uniqueness gives $A=A^*$ and then Lemma \ref{lem:fAB} implies $d=d^*$.

($b$) Apply Lemma \ref{lem:level-matching} by letting $\kappa_1=-K_{\ell}$ and $\kappa_2=K_h$. Since both inequalities $-\frac{c}{\beta}<\kappa_1$ and $\kappa_2<\frac{r\theta(\mu_h-\mu_\ell)}{\beta(\beta+\theta)}-\frac{c}{\beta}$ are strict, the lemma yields a unique pair $A^*,B^*>0$ with $F_{A^*,B^*}\bigl(d(A^*,B^*)\bigr)=-K_{\ell}$ and $F_{A^*,B^*}\bigl(u(A^*,B^*)\bigr)=K_h$. Set
\[
d^*=d(A^*,B^*),
\qquad
u^*=u(A^*,B^*).
\]
Lemma \ref{lem:fAB} implies that both numbers are finite, that $d^*<u^*$, and that $f_{A^*,B^*}(d^*)=f_{A^*,B^*}(u^*)=0$.
This proves existence.

For uniqueness, let $(d,u,A,B)$ be any finite solution. The value-matching equations $F_{A,B}(d)=-K_{\ell}$ and $F_{A,B}(u)=K_h$ imply $d\neq u$. Moreover, $A\leq0$ and $B\leq0$ cannot hold simultaneously, because they would imply
\[
f_{A,B}(x)
=
J'(x)+A(\psi_-^h)'(x)-B(\psi_+^\ell)'(x)>0
\]
for every $x$, contradicting $f_{A,B}(d)=0$. Thus Lemma \ref{lem:minimizer} applies. It implies that $f_{A,B}$ is strictly positive between the two distinct zeros $d$ and $u$; otherwise there would be a nonpositive local minimum between them.
Since
\[
F_{A,B}(u)-F_{A,B}(d)=K_{\ell}+K_h>0,
\]
we must have $d<u$ and $\sup_{x\in\R}f_{A,B}(x)>0$.

Lemma \ref{lem:A>0,B>0} now gives $A,B>0$. Lemma \ref{lem:fAB} therefore identifies the zeros as $d=d(A,B)$ and $u=u(A,B)$. The uniqueness assertion in Lemma \ref{lem:level-matching} then implies $A=A^*$, $B=B^*$, and consequently $d=d^*$ and $u=u^*$.

($c$) Apply Lemma \ref{lem:level-matching} by letting $\kappa_1=-\frac{c}{\beta}$ and $\kappa_2=K_h$. It yields a pair $(0,B^*)$ with $B^*>0$ and a finite zero $u^*=u(0,B^*)$ satisfying
\[
F_{0,B^*}(u^*)=K_h,
\qquad
f_{0,B^*}(u^*)=0.
\]
This proves existence.

Let $(u,B)$ be any finite solution of this system. The equation $f_{0,B}(u)=0$ implies
\[
B=\frac{J'(u)}{(\psi_+^\ell)'(u)}>0.
\]
Since $\psi_+^\ell(x)\to0$ as $x\to-\infty$,
\[
\lim_{x\to-\infty}F_{0,B}(x)=\lim_{x\to-\infty} [-B\psi_+^{\ell}(x)+J(x)]=-\frac{c}{\beta},
\]
where the last equality follows from \eqref{eq:limit-psi-ell-1} and \eqref{eq:lim-J}. 
Therefore,
\[
\int_{-\infty}^u f_{0,B}(x)\,\mathrm{d}x
=F_{0,B}(u)-\left(-\frac{c}{\beta}\right)
=K_h+\frac{c}{\beta}
>0.
\]
Hence $\sup_{x\in\mathbb{R}}f_{0,B}(x)>0$. Lemma \ref{lem:fAB} implies that $u=u(0,B)$ is its unique zero and $d(0,B)=-\infty$.
Thus $(0,B)$ satisfies the conditions in \eqref{eq:F-kappa1-kappa2} of Lemma \ref{lem:level-matching} with $(\kappa_1,\kappa_2)=(-\frac{c}{\beta},K_h)$.
Uniqueness gives $B=B^*$ and then $u=u^*$.
\end{proof}

\section{Optimality of Proposed Policies}
\label{sec:optimality}

In this section, we prove Theorem~\ref{thm:main} by verification. We first establish a verification theorem in 
\S\ref{sub-sec:verification} and then apply it to the five proposed policies in \S\ref{sub-sec:optimality}.

\subsection{Verification Theorem}
\label{sub-sec:verification}

We denote a candidate function by $\widehat{V}$ and its first and second partial derivatives with respect to $x$ by $\widehat{V}_x(x,\mu)$ and $\widehat{V}_{xx}(x,\mu)$, respectively. Throughout this section, sums involving switching costs are taken only over finite switching times. 

\begin{proposition}[Verification Theorem]
\label{prop:verfication}
Let $\widehat{V}(\cdot,\mu_h)$ and $\widehat{V}(\cdot,\mu_\ell)$ be continuously differentiable functions on $\mathbb R$ whose
first derivatives are globally Lipschitz continuous. Suppose that
\begin{align}
\Gamma_h\widehat{V}(x,\mu_h)-\beta \widehat{V}(x,\mu_h)  +rx^+-c&\leq0, 
\label{eq:verification-inequality-1}\\
\Gamma_\ell \widehat{V}(x,\mu_\ell)-\beta \widehat{V}(x,\mu_\ell)
  +rx^+&\leq0
\label{eq:verification-inequality-2}
\end{align}
for almost every $x\in\mathbb R$, and that
\begin{equation}
\label{eq:verification-inequality-3}
-K_{\ell}\leq \widehat{V}(x,\mu_h)-\widehat{V}(x,\mu_\ell)\leq K_h
\qquad\text{for every }x\in\mathbb R.
\end{equation}
Then, for every initial state and mode $(x,\mu)$ and every admissible policy $\pi\in\Pi$,
\[
\mathcal J((x,\mu),\pi)\leq \widehat{V}(x,\mu).
\]
\end{proposition}

We use the following uniform moment bound; its proof is given in Appendix \ref{app-sub:lem-bound}.

\begin{lemma}
\label{lem:bound}
For every initial state and mode $(x,\mu)$ and every admissible policy,
\begin{equation}
\label{eq:bound-lemma}
\mathbb E_{(x,\mu)}[(X(t))^2]
\leq3\left(x^2+\bar\mu^2+\frac{\sigma^2}{2\theta}\right),
\qquad t\geq0,
\end{equation}
where $\bar\mu:=\max\{|\mu_h|,|\mu_\ell|\}$.
\end{lemma}

\begin{proof}[Proof of Proposition~\ref{prop:verfication}]
Fix $(x,\mu)$ and an admissible policy $\pi$. Global Lipschitz continuity of the first derivatives implies that they are locally absolutely continuous and that the second derivatives exist almost everywhere and are essentially bounded. There exists a constant $C>0$ such that, for both modes,
\begin{equation}
\label{eq:bound-V}
|\widehat{V}_x(y,\mu)|\leq C(1+|y|),
\qquad
|\widehat{V}(y,\mu)|\leq C(1+y^2),
\qquad y\in\mathbb R.
\end{equation}
Lemma~\ref{lem:bound} consequently gives
\[
\E_{(x,\mu)}\left[|\widehat{V}(X(t),\mu(t))|\right]
\leq C \left(1+\E_{(x,\mu)}[(X(t))^2]\right)
\leq C \left(1+3\left(x^2+\bar{\mu}^2+\frac{\sigma^2}{2\theta}\right)\right)
\]
and then
\begin{equation}
\label{eq:verification-transversality}
\lim_{t\to\infty}
\mathbb E_{(x,\mu)}
\left[e^{-\beta t}|\widehat{V}(X(t),\mu(t))|\right]=0.
\end{equation}
We also have 
\begin{align}
&\mathbb E_{(x,\mu)}\left[ \int_0^\infty e^{-\beta s} \left|rX(s)^+-c\mathbf1_{\{\mu(s)=\mu_h\}}\right| \,\mathrm ds\right] \nonumber\\
&\quad =  \int_0^\infty e^{-\beta s} \mathbb E_{(x,\mu)}\left[\left|rX(s)^+-c\mathbf1_{\{\mu(s)=\mu_h\}}\right| \right]\,\mathrm ds \nonumber\\
&\quad\leq \int_0^\infty e^{-\beta s} \left(r\mathbb E_{(x,\mu)}\left[|X(s)|\right]+c\right) \,\mathrm ds \nonumber\\
&\quad\leq \int_0^\infty e^{-\beta s} \left[r\sqrt{3\left(x^2+\bar\mu^2+\frac{\sigma^2}{2\theta}\right)}+c\right]\,\mathrm ds \nonumber\\
&\quad=\frac1\beta\left[r\sqrt{3\left(x^2+\bar\mu^2+\frac{\sigma^2}{2\theta}\right)}+c\right]\nonumber\\
&\quad<\infty,\label{eq:verification-running-integrability}
\end{align}
where the first equality follows from Tonelli's Theorem, the first inequality follows from 
\[
\left|rX(s)^+-c\mathbf1_{\{\mu(s)=\mu_h\}}\right|\leq rX(s)^++c\leq r|X(s)|+c,
\]
the second inequality follows from $\mathbb E_{(x,\mu)}\left[|X(s)|\right]\leq \sqrt{\mathbb E_{(x,\mu)}\left[(X(s))^2\right]}$ (Cauchy-Schwarz inequality) and \eqref{eq:bound-lemma}.

For $t\geq0$, define
\[
C_t:=\sum_{\{n:0\leq\tau_n\leq t\}}
e^{-\beta\tau_n}K(\mu(\tau_n-),\mu(\tau_n)),
\qquad C_\infty:=\lim_{t\to\infty}C_t.
\]
If $\mathbb E_{(x,\mu)}[C_\infty]=\infty$, then \eqref{eq:verification-running-integrability} implies 
$\mathcal J((x,\mu),\pi)=-\infty$, and the conclusion is immediate. 
It therefore suffices to consider policies for which $\mathbb E_{(x,\mu)}[C_\infty]<\infty$.

Since the first derivatives are locally absolutely continuous, the generalized It\^o formula applies on each interval between
switches; see, for example, Lemma~3.1 in \cite{DaiYao2013Part2} for the discounted formulation, or Theorem~71 in \cite{Protter2005}.
Applying this formula between successive switching times and including the possible initial switch gives
\begin{align}
\nonumber
e^{-\beta t}\widehat{V}(X(t),\mu(t)) 
=&\widehat{V}(x,\mu)+\int_0^t e^{-\beta s}
\Big[
\frac{\sigma^2}{2}\widehat{V}_{xx}(X(s),\mu(s))
\\[-2pt]
&
+\theta(\mu(s)-X(s))\widehat{V}_x(X(s),\mu(s))
-\beta \widehat{V}(X(s),\mu(s))
\Big]\,\mathrm ds
\nonumber\\
& +M_t
+\sum_{\{n:0\leq\tau_n\leq t\}}e^{-\beta\tau_n}
\Big[
\widehat{V}(X(\tau_n),\mu(\tau_n))
-\widehat{V}(X(\tau_n),\mu(\tau_n-))
\Big],\label{eq:Ito}
\end{align}
where $M_t:=\sigma\int_0^t e^{-\beta s}\widehat{V}_x(X(s),\mu(s-))\,\mathrm dB(s)$.
By \eqref{eq:bound-V} and Lemma~\ref{lem:bound},
\begin{align*}
&\mathbb E_{(x,\mu)}\left[\int_0^t e^{-2\beta s}|\widehat{V}_x(X(s),\mu(s-))|^2\,\mathrm ds\right]
\\ &\quad\leq 2C^2\int_0^t e^{-2\beta s} \left[1+3\left(x^2+\bar\mu^2+\frac{\sigma^2}{2\theta}\right)\right]\,\mathrm ds<\infty.
\end{align*}
Hence $M$ is a square-integrable martingale and $\mathbb E_{(x,\mu)}[M_t]=0$.

At every switching time, \eqref{eq:verification-inequality-3} implies
\[
\widehat{V}(X(\tau_n),\mu(\tau_n))-\widehat{V}(X(\tau_n),\mu(\tau_n-))
\leq K(\mu(\tau_n-),\mu(\tau_n)).
\]
Combining this inequality with the generator inequalities in \eqref{eq:Ito}, we obtain
\begin{align*}
e^{-\beta t}\widehat{V}(X(t),\mu(t))
\leq \widehat{V}(x,\mu)
-\int_0^t e^{-\beta s}
\left[rX(s)^+-c\mathbf1_{\{\mu(s)=\mu_h\}}\right]\,\mathrm ds
+M_t+C_t.
\end{align*}
All terms are integrable under the present finite-cost assumption. Taking expectations and rearranging gives
\begin{align*}
\widehat{V}(x,\mu)
\geq\mathbb E_{(x,\mu)}\left[
\int_0^t e^{-\beta s}
\left[rX(s)^+-c\mathbf1_{\{\mu(s)=\mu_h\}}\right]\,\mathrm ds
-C_t+e^{-\beta t}\widehat{V}(X(t),\mu(t))
\right].
\end{align*}
Letting $t\to\infty$, we use \eqref{eq:verification-running-integrability} for convergence of the running-profit integral in $\mathcal{L}^1$, monotone convergence for $C_t$, and \eqref{eq:verification-transversality} for the terminal term. We conclude that
\[
\widehat{V}(x,\mu)\geq\mathcal J((x,\mu),\pi).
\]
Since $\pi$ was arbitrary, the proposition follows.
\end{proof}

\subsection{Optimality of Proposed Policies}
\label{sub-sec:optimality}

For each case of Theorem~\ref{thm:main}, we show that the profit function of the proposed policy satisfies the hypotheses of Proposition~\ref{prop:verfication}. It therefore dominates the profit of every admissible policy. Since it is attained by the proposed policy, it equals the value function.

\begin{proof}[Proof of Theorem~\ref{thm:main}]

By the definition of admissibility, the proposed policies in five cases in Theorem \ref{thm:main} are admissible. We next check the conditions in Proposition \ref{prop:verfication}.

Before proving the the optimality of proposed candidate function, we first denote it as $\widehat{V}$. 
First we show that all candidate functions in five cases are continuously differentiable and their first derivatives are globally Lipschitz continuous. For the no-switching policies in parts ($a$) and ($c$), the continuous differentiability and Lipschitz continuity follow from $J_h$ and $J_{\ell}$; see Lemma \ref{lem:J}. For the threshold policies in parts ($b$), ($d$), and ($e$), continuity of the functions $\widehat{V}$ and their first derivatives at the thresholds follow from the value-matching and smooth-fit equations in
Proposition~\ref{prop:existence}. 

For the Lipschitz continuity, we take part ($b$) as an example, and the cases in parts ($d$) and ($e$) are similar. In part ($b$), we have
$\widehat{V}(x,\mu_\ell)=\mathcal{J}((x,\mu_\ell),(d^*,\infty))=J_{\ell}(x)$ and 
\begin{align*}
    \widehat{V}(x,\mu_h)=\mathcal{J}((x,\mu_h),(d^*,\infty))=
    \begin{cases}
        J_{\ell}(x)-K_{\ell} &\text{for $x\leq d^*$},\\
        J_h(x) + \frac{-J(d^*)-K_{\ell}}{\psi_{-}^h(d^*)}\psi_{-}^h(x) &\text{for $x>d^*$}.
    \end{cases}
\end{align*}
The first derivative of function $\widehat{V}(x,\mu_\ell)=J_\ell(x)$ is globally Lipschitz continuous from Lemma \ref{lem:J}. 
For $\widehat{V}(x,\mu_h)$, its derivative is Lipschitz on each side of $d^*$ and continuous at $d^*$ (by Lemma~\ref{lem:J} (b) and Lemma~\ref{lem:psi} ($d$)); hence it is globally Lipschitz.

It remains to check the three verification inequalities in \eqref{eq:verification-inequality-1}-\eqref{eq:verification-inequality-3}.
We will use the following equations, obtained from the ODEs for $J$, $\psi_-^h$, and $\psi_+^\ell$:
\begin{align}
\Gamma_h F_{A^*,B^*}(x)-\beta F_{A^*,B^*}(x)
&=c-\theta(\mu_h-\mu_\ell)\left[J_\ell'(x)+B^*(\psi_+^\ell)'(x)\right],\label{eq:verification-F-h}\\
\Gamma_\ell F_{A^*,B^*}(x)-\beta F_{A^*,B^*}(x)
&=c-\theta(\mu_h-\mu_\ell)\left[J_h'(x)+A^*(\psi_-^h)'(x)\right],\label{eq:verification-F-ell}
\end{align}
where $F$ is defined in \eqref{eq:F}, the first equation is from \eqref{eq:ODE-Jh-Jl-h}, \eqref{eq:ODE-psi-h}, and \eqref{eq:ODE-psi-ell-with-h}, and the second one is from  \eqref{eq:ODE-Jh-Jl-ell}, \eqref{eq:ODE-psi-l}, and \eqref{eq:ODE-psi-h-with-ell}.

In view of the fact that $F_{A^*,B^*}'=f_{A^*,B^*}$ (see \eqref{eq:F}-\eqref{eq:f}), whenever $F_{A^*,B^*}(d^*)=-K_{\ell}$ and $f_{A^*,B^*}(d^*)=0$, the equation \eqref{eq:verification-F-h} yields
\begin{equation}
\label{eq:verification-boundary-d}
\theta(\mu_h-\mu_\ell)\left[J_\ell'(d^*)+B^*(\psi_+^\ell)'(d^*)\right]+\beta K_{\ell}-c
=-\frac{\sigma^2}{2}f_{A^*,B^*}'(d^*).
\end{equation}
Similarly, whenever $F_{A^*,B^*}(u^*)=K_h$ and $f_{A^*,B^*}(u^*)=0$, the equation \eqref{eq:verification-F-ell} yields
\begin{equation}
\label{eq:verification-boundary-u}
c+\beta K_h-\theta(\mu_h-\mu_\ell)\left[J_h'(u^*)+A^*(\psi_-^h)'(u^*)\right]=\frac{\sigma^2}{2}f_{A^*,B^*}'(u^*).
\end{equation}

($a$) In this case, the parameter condition is $\frac{r\theta(\mu_{h}-\mu_{\ell})}{\beta(\beta+\theta)}-\frac{c}{\beta}\leq -K_{\ell}$, and the proposed no-switching policy that permanently adopts the lower mode has profit functions $\widehat{V}(x,\mu_{h})=J_\ell(x)-K_{\ell}$ and $\widehat{V}(x,\mu_{\ell})=J_\ell(x)$. 

By Lemma \ref{lem:J}, we have $J_\ell''>0$, $\lim_{x\to-\infty}J_{\ell}'(x)=0$, and $\lim_{x\to\infty}J_{\ell}'(x)=\frac{r}{\beta+\theta}$. Then, \eqref{eq:ODE-J-l-Gamma h} and $0<J_\ell'(x)<\frac{r}{\beta+\theta}$ give
\begin{align*}
\Gamma_h\widehat{V}(x,\mu_h)
-\beta \widehat{V}(x,\mu_h)+rx^+-c
&=\theta(\mu_h-\mu_\ell)J_\ell'(x)+\beta K_{\ell}-c\\
&\leq\frac{r\theta(\mu_h-\mu_\ell)}{\beta+\theta}
+\beta K_{\ell}-c
\leq0,
\end{align*}
where the last inequality is the parameter condition. We have proven \eqref{eq:verification-inequality-1}.

The ODE \eqref{eq:ODE-J-l} gives $\Gamma_\ell \widehat{V}(x,\mu_\ell)-\beta \widehat{V}(x,\mu_\ell)+rx^+=0$, proving \eqref{eq:verification-inequality-2}.

Finally, the difference of $\widehat{V}(x,\mu_{h})=J_\ell(x)-K_{\ell}$ and $\widehat{V}(x,\mu_{\ell})=J_\ell(x)$ is $-K_{\ell}$, 
so \eqref{eq:verification-inequality-3} holds.

($b$) In this case, the parameter condition is $-\frac{c}{\beta}<-K_{\ell} < \frac{r\theta(\mu_{h}-\mu_{\ell})}{\beta(\beta+\theta)}-\frac{c}{\beta}\leq K_h$, and the proposed $(d^*,\infty)$ policy has profit functions
\[
\widehat{V}(x,\mu_h)=
\begin{cases}
J_\ell(x)-K_{\ell},&x\leq d^*,\\
J_h(x)+A^*\psi_-^h(x),&x>d^*,
\end{cases}
\qquad
\widehat{V}(x,\mu_\ell)=J_\ell(x),
\]
where $(d^*,A^*)$ is the unique solution of \eqref{eq:optimal-equation-d} as stated in Proposition \ref{prop:existence} ($a$), and 
$A^*=\frac{-J(d^*)-K_{\ell}}{\psi_{-}^h(d^*)}$ from \eqref{eq:A-2}.

First, we check \eqref{eq:verification-inequality-1}. For $x>d^*$, the ODEs in \eqref{eq:ODE-J-h} for $J_h$ and in \eqref{eq:ODE-psi-h} for $\psi_{-}^h$ imply that $\Gamma_h \widehat{V}(x,\mu_h) -\beta \widehat{V}(x,\mu_h)+rx^+-c=0$.
For $x< d^*$, ODE \eqref{eq:ODE-J-l-Gamma h} implies 
\begin{align*}
\Gamma_h \widehat{V}(x,\mu_h)
-\beta \widehat{V}(x,\mu_h)+rx^+-c
&=\theta(\mu_h-\mu_\ell)J_\ell'(x)+\beta K_{\ell}-c\\
&\leq\theta(\mu_h-\mu_\ell)J_\ell'(d^*)+\beta K_{\ell}-c\\
&=-\frac{\sigma^2}{2}f_{A^*,0}'(d^*)
<0,
\end{align*}
where the first inequality follows from $J_\ell''>0$ (see Lemma \ref{lem:J}), the second equality follows from \eqref{eq:verification-boundary-d} with $B^*=0$, and the second inequality follows from Lemma~\ref{lem:fAB}.

The verification inequality \eqref{eq:verification-inequality-2} holds because $J_\ell$ satisfies \eqref{eq:ODE-J-l}.

Finally, the difference of $\widehat{V}(x,\mu_{h})$ and $\widehat{V}(x,\mu_{\ell})$ is $-K_{\ell}$ for $x\leq d^*$ and $F_{A^*,0}(x)$ for $x>d^*$.
The limits \eqref{eq:limit-psi-h-2} and \eqref{eq:lim-J} imply 
\[
\lim_{x\to\infty}F_{A^*,0}(x)=\lim_{x\to\infty}\left[A^*\psi_-^h(x)+J(x)\right]
=\frac{r\theta(\mu_h-\mu_\ell)}{\beta(\beta+\theta)}
-\frac c\beta\leq K_h.
\]
where the inequality follows from parameter condition in this case. 

Thus, $F_{A^*,0}'=f_{A^*,0}>0$ on $(d^*,\infty)$ (see Lemma~\ref{lem:fAB}) and $F_{A^*,0}(d^*)=-K_{\ell}$ imply that $-K_{\ell}<F_{A^*,0}(x)<K_h$ for $x>d^*$, and then \eqref{eq:verification-inequality-3} holds.

($c$) In this case, the parameter condition is $-K_{\ell}\leq -\frac{c}{\beta} < \frac{r\theta(\mu_{h}-\mu_{\ell})}{\beta(\beta+\theta)}-\frac{c}{\beta}\leq K_h$, and the proposed no-switching policy that permanently adopts the initial mode has profit functions $\widehat{V}(x,\mu_h)=J_h(x)$ and $\widehat{V}(x,\mu_\ell)=J_\ell(x)$.

Both inequalities \eqref{eq:verification-inequality-1} and \eqref{eq:verification-inequality-2} hold with equality by the ODEs for $J_h$ and $J_\ell$. Furthermore, the parameter conditions and the limits and monotonicity of $J$ (see Lemma \ref{lem:J}) give
\[
-K_{\ell}\leq-\frac c\beta<\widehat{V}(x,\mu_h)-\widehat{V}(x,\mu_\ell)=J(x)
<\frac{r\theta(\mu_h-\mu_\ell)}{\beta(\beta+\theta)}-\frac c\beta\leq K_h.
\]
Hence \eqref{eq:verification-inequality-3} also holds.

($d$)  In this case, the parameter condition is $-\frac{c}{\beta}< -K_{\ell} <K_h< \frac{r\theta(\mu_{h}-\mu_{\ell})}{\beta(\beta+\theta)}-\frac{c}{\beta}$, and the proposed $(d^*,u^*)$ policy has profit functions
\[
\widehat{V}(x,\mu_h)=
\begin{cases}
J_\ell(x)+B^*\psi_+^\ell(x)-K_{d},&x\leq d^*,\\
J_h(x)+A^*\psi_-^h(x),&x>d^*,
\end{cases}
\]
and
\[
\widehat{V}(x,\mu_\ell)=
\begin{cases}
J_\ell(x)+B^*\psi_+^\ell(x),&x<u^*,\\
J_h(x)+A^*\psi_-^h(x)-K_h,&x\geq u^*,
\end{cases}
\]
where $(d^*,u^*,A^*,B^*)$ is the unique solution of \eqref{eq:optimal-equation-du-1}-\eqref{eq:optimal-equation-du-2}, i.e., 
\[
F_{A^*,B^*}(d^*)=-K_{\ell},
\qquad F_{A^*,B^*}(u^*)=K_h,
\qquad f_{A^*,B^*}(d^*)=f_{A^*,B^*}(u^*)=0.
\]

First, the verification inequality \eqref{eq:verification-inequality-1} holds for $x>d^*$ with equality by the ODEs for $J_h$ in \eqref{eq:ODE-J-h} and $\psi_{-}^h$ in \eqref{eq:ODE-psi-h}.
For $x< d^*$, ODEs \eqref{eq:ODE-J-l-Gamma h} and \eqref{eq:ODE-psi-ell-with-h} imply 
\begin{align*}
\Gamma_h \widehat{V}(x,\mu_h)-\beta \widehat{V}(x,\mu_h)+rx^+-c
&=\theta(\mu_h-\mu_\ell)\left[J_\ell'(x)+B^*(\psi_+^\ell)'(x)\right]+\beta K_{\ell}-c\\
&\leq \theta(\mu_h-\mu_\ell)\left[J_\ell'(d^*)+B^*(\psi_+^\ell)'(d^*)\right]+\beta K_{\ell}-c\\
&=-\frac{\sigma^2}{2}f_{A^*,B^*}'(d^*)
<0,
\end{align*}
where the first inequality follows from $J_\ell''+B^*(\psi_+^\ell)''>0$ (due to $J_\ell''>0$ by Lemma \ref{lem:J} ($b$), $(\psi_+^\ell)''>0$ by Lemma \ref{lem:psi} ($b$), and $B^*>0$), the last equality follows from \eqref{eq:verification-boundary-d}, and the last inequality follows from $f_{A^*,B^*}'(d^*)>0$ in Lemma~\ref{lem:fAB}.

Similarly, the verification inequality \eqref{eq:verification-inequality-2} holds for $x<u^*$ with equality by the ODEs for $J_\ell$ in \eqref{eq:ODE-J-l} and $\psi_{+}^\ell$ in \eqref{eq:ODE-psi-l}. 
For $x> u^*$, ODEs \eqref{eq:ODE-J-h-Gamma ell} and \eqref{eq:ODE-psi-h-with-ell} imply 
\begin{align*}
\Gamma_\ell \widehat{V}(x,\mu_\ell) -\beta \widehat{V}(x,\mu_\ell)+rx^+
&=c+\beta K_h-\theta(\mu_h-\mu_\ell)\left[J_h'(x)+A^*(\psi_-^h)'(x)\right]\\
&\leq c+\beta K_h-\theta(\mu_h-\mu_\ell)\left[J_h'(u^*)+A^*(\psi_-^h)'(u^*)\right]\\
&=\frac{\sigma^2}{2}f_{A^*,B^*}'(u^*)
<0,
\end{align*}
where we used $J_h''+A^*(\psi_-^h)''>0$, \eqref{eq:verification-boundary-u}, and Lemma~\ref{lem:fAB}. 

Finally, 
\[
\widehat{V}(x,\mu_h)-\widehat{V}(x,\mu_\ell)
=\begin{cases}
-K_{\ell},&x\leq d^*,\\
F_{A^*,B^*}(x),&d^*<x<u^*,\\
K_h,&x\geq u^*.
\end{cases}
\]
Since $F_{A^*,B^*}'=f_{A^*,B^*}>0$ on $(d^*,u^*)$, $F_{A^*,B^*}$ increases strictly from $-K_{\ell}$ to $K_h$ on $[d^*,u^*]$.
Thus \eqref{eq:verification-inequality-3} holds.

($e$) In this case, the parameter condition is $-K_{\ell}\leq-\frac{c}{\beta} <K_h< \frac{r\theta(\mu_{h}-\mu_{\ell})}{\beta(\beta+\theta)}-\frac{c}{\beta}$, and the proposed  $(-\infty,u^*)$ policy has profit functions
\[
\widehat{V}(x,\mu_h)=J_h(x),
\qquad
\widehat{V}(x,\mu_\ell)=
\begin{cases}
J_\ell(x)+B^*\psi_+^\ell(x),&x<u^*,\\
J_h(x)-K_h,&x\geq u^*.
\end{cases}
\]

The ODE \eqref{eq:ODE-J-h} gives $\Gamma_h \widehat{V}(x,\mu_h)-\beta \widehat{V}(x,\mu_h)+rx^+-c=0$,proving \eqref{eq:verification-inequality-1}. 

For $x<u^*$, the ODEs in \eqref{eq:ODE-J-l} for $J_\ell$ and in \eqref{eq:ODE-psi-l} for $\psi_+^\ell$ imply that
$\Gamma_\ell \widehat{V}(x,\mu_\ell) -\beta \widehat{V}(x,\mu_\ell)+rx^+=0$. 
For $x> u^*$, ODE \eqref{eq:ODE-J-h-Gamma ell} implies 
\begin{align*}
\Gamma_\ell \widehat{V}(x,\mu_\ell)-\beta \widehat{V}(x,\mu_\ell)+rx^+
&=c+\beta K_h-\theta(\mu_h-\mu_\ell)J_h'(x)\\
&\leq c+\beta K_h-\theta(\mu_h-\mu_\ell)J_h'(u^*)\\
&=\frac{\sigma^2}{2}f_{0,B^*}'(u^*)
<0,
\end{align*}
using $J_h''>0$, \eqref{eq:verification-boundary-u} with $A^*=0$, and Lemma~\ref{lem:fAB}. We have proven \eqref{eq:verification-inequality-2}.

Finally, the difference between $\widehat{V}(x,\mu_h)$ and $\widehat{V}(x,\mu_\ell)$ is $F_{0,B^*}(x)$ for $x<u^*$ and $K_h$ for
$x\geq u^*$. Lemma~\ref{lem:fAB} gives $f_{0,B^*}>0$ on $(-\infty,u^*)$, and
\[
\lim_{x\to-\infty}F_{0,B^*}(x)=\lim_{x\to-\infty}\left[-B^*\psi_+^\ell(x)+J(x)\right]=-\frac c\beta\geq-K_{\ell},
\]
where the second equality follows from \eqref{eq:lim-J} and \eqref{eq:limit-psi-ell-1}, and the inequality follows from the parameter condition in this case. Thus $-K_{\ell}<F_{0,B^*}(x)<K_h$ for $x<u^*$, proving \eqref{eq:verification-inequality-3}.

\medskip
In every case, the candidate profit functions satisfy all hypotheses of Proposition~\ref{prop:verfication}.
Therefore, for every admissible policy $\pi$,
\begin{equation}
\label{eq:J<V}
\mathcal J((x,\mu),\pi)\leq \widehat{V}(x,\mu).
\end{equation}
Since the proposed admissible policy $\pi^*$ attains the candidate, i.e.,
\[
\widehat{V}(x,\mu)=\mathcal J((x,\mu),\pi^*)\leq V(x,\mu),
\]
taking the supremum in \eqref{eq:J<V} yields $V=\widehat{V}$. 
Hence $V$ is the value function and $\pi^*$ is optimal. The uniqueness of the finite thresholds was established in
Proposition~\ref{prop:existence}. This completes the proof.
\end{proof}

\section{Concluding Remarks}
\label{sec:conclusion} 

This paper studies an infinite-horizon discounted optimal switching problem in which the mean-reversion level of an Ornstein--Uhlenbeck
process can be switched between two values. The model combines state-dependent revenue, an additional running cost in the high mode, and positive, potentially asymmetric switching costs. We characterize the value function and identify five optimal policy regimes, including
permanent operation in one mode and policies with one or two finite switching thresholds. The classification is determined by comparing
the limiting profit differences between the two permanent modes with the switching costs. For the threshold regimes, we establish the existence and uniqueness of the finite thresholds through the value-matching and smooth-fit equations, and derive representations of the associated profit functions. A verification argument then establishes optimality over the full class of admissible switching policies. These results clarify how the benefits of a higher mean-reversion level, its additional running cost, and switching costs jointly determine whether switching is worthwhile and whether it occurs in one or both directions. Extensions to mode-dependent volatility, more than two operating modes, and alternative revenue functions provide directions for future research.

\bibliographystyle{ormsv080}
\bibliography{refs}

@article{cadenillas2013optimal,
  title={Optimal production management when demand depends on the business cycle},
  author={Cadenillas, Abel and Lakner, Peter and Pinedo, Michael},
  journal={Operations Research},
  volume={61},
  number={4},
  pages={1046--1062},
  year={2013},
  publisher={INFORMS}
}

@article{cadenillas2010optimal,
  title={Optimal control of a mean-reverting inventory},
  author={Cadenillas, Abel and Lakner, Peter and Pinedo, Michael},
  journal={Operations Research},
  volume={58},
  number={6},
  pages={1697--1710},
  year={2010},
  publisher={INFORMS}
}

@article{ZhangZhang2008,
    author = {Zhang, Hanqin and Zhang, Qing},
    title = {Trading a mean-reverting asset: Buy low and sell high},
    journal = {Automatica},
    year = {2008},
    volume = {44}, 
    number = {6},
    pages = {1511--1518}
}

@article{Suzuki2021,
    author = {Suzuki, Kiyoshi},
    title = {Infinite-horizon optimal switching regions for a pair-trading strategy with quadratic risk aversion considering simultaneous multiple switchings: A viscosity solution approach},
    journal = {Mathematics of Operations Research},
    year = {2021},
    volume = {46}, 
    number = {1},
    pages = {336-360}
}

@article{WuChao2014,
  title={Optimal control of a {B}rownian production/inventory system with average cost criterion},
  author={Wu, Jingchen and Chao, Xiuli},
  journal={Mathematics of Operations Research},
  volume={39},
  number={1},
  pages={163--189},
  year={2014},
  publisher={INFORMS}
}

@article{ngo2016optimal,
  title={Optimal switching for the pairs trading rule: A viscosity solutions approach},
  author={Ngo, Minh-Man and Pham, Huy\^{e}n},
  journal={Journal of Mathematical Analysis and Applications},
  volume={441},
  number={1},
  pages={403--425},
  year={2016},
  publisher={Elsevier}
}

@article{duckworth2001model,
  title={A model for investment decisions with switching costs},
  author={Duckworth, Kate and Zervos, Mihail},
  journal={The Annals of Applied Probability},
  pages={239--260},
  year={2001},
  volume={11},
  number={1},
  publisher={JSTOR}
}

@article{pham2009optimal,
  title={Optimal switching over multiple regimes},
  author={Pham, Huy\^{e}n and Ly Vath, Vathana and Zhou, Xun Yu},
  journal={SIAM Journal on Control and Optimization},
  volume={48},
  number={4},
  pages={2217--2253},
  year={2009},
  publisher={SIAM}
}

@article{bayraktar2010one,
  title={On the one-dimensional optimal switching problem},
  author={Bayraktar, Erhan and Egami, Masahiko},
  journal={Mathematics of Operations Research},
  volume={35},
  number={1},
  pages={140--159},
  year={2010},
  publisher={INFORMS}
}

@article{KanoriaLobelLu2024,
    author = {Kanoria, Yash and Lobel, Ilan and Lu, Jiaqi},
    title = {Managing customer churn via service mode control},
    journal = {Mathematics of Operations Research},
    year = {2024},
    volume = {49},
    number = {2},
    pages ={1192-1222}
}

@article{dayanik2003optimal,
  title={On the optimal stopping problem for one-dimensional diffusions},
  author={Dayanik, Savas and Karatzas, Ioannis},
  journal={Stochastic Processes and their Applications},
  volume={107},
  number={2},
  pages={173--212},
  year={2003},
  publisher={Elsevier}
}

@article{OrmeciMatogluVandeVate2011,
    author = {Ormeci Matoglu, Melda and Vande Vate, John},
    title = {Drift control with changeover costs},
    journal = {Operations Research},
    year = {2011},
    volume = {59},
    number = {2},
    pages = {427--439}
}

@article{Doshi1978,
    author = {Doshi, Bharat T.},
    title = {Two-mode control of Brownian motion with quadratic loss and switching costs},
    journal = {Stochastic Processes and their Applications},
    year = {1978},
    volume = {6},
    number = {3},
    pages = {277--289}
}

@article{Zhu2011,
    author = {Zhu, Chao},
    title = {Optimal control of the risk process in a regime-switching environment},
    journal = {Automatica},
    year = {2011},
    volume = {47},
    number = {8},
    pages = {1570-1579}
}

@article{Dixit1989,
    author = {Dixit, Avinash},
    title = {Entry and exit decisions under uncertainty},
    journal = {Journal of Political Economy},
    year = {1989},
    volume = {97},
    number = {3},
    pages = {620--638}
}

@article{BrekkeOksendal1994,
    author = {Brekke, Kjell Arne and {\O}ksendal, Bernt},
    title = {Optimal switching in an economic activity under uncertainty},
    journal = {SIAM Journal on Control and Optimization},
    year = {1994},
    volume = {32},
    number = {4},
    pages = {1021--1036}
}

@article{LyVathPham2007,
    author = {Ly Vath, Vathana and Pham, Huy{\^e}n},
    title = {Explicit solution to an optimal switching problem in the two-regime case},
    journal = {SIAM Journal on Control and Optimization},
    year = {2007},
    volume = {46},
    number = {2},
    pages = {395--426}
}

@article{Suzuki2016,
    author = {Suzuki, Kiyoshi},
    title = {Optimal switching strategy of a mean-reverting asset over multiple regimes},
    journal = {Automatica},
    year = {2016},
    volume = {67},
    pages = {33-45}
}

@article{EloeLYYZ2008,
  author    = {Eloe, P. and Liu, R. H. and Yatsuki, M. and Yin, G. and Zhang, Q.},
  title     = {Optimal Selling Rules in a Regime-Switching Exponential Gaussian Diffusion Model},
  journal   = {SIAM Journal on Applied Mathematics},
  year      = {2008},
  volume    = {69},
  number    = {3},
  pages     = {810--829},
  doi       = {10.1137/060652671}
}

@incollection{Pham2007,
  author    = {Pham, Huy\^{e}n},
  title     = {On the Smooth-Fit Property for One-Dimensional Optimal Switching Problem},
  booktitle = {S{\'e}minaire de Probabilit{\'e}s XL},
  editor    = {Donati-Martin, Catherine and {\'E}mery, Michel and Rouault, Alain and Stricker, Christophe},
  series    = {Lecture Notes in Mathematics},
  volume    = {1899},
  pages     = {187--199},
  year      = {2007},
  publisher = {Springer},
  address   = {Berlin, Heidelberg},
}

@book{karatzas1991brownian,
  title={Brownian Motion and Stochastic Calculus},
  author={Karatzas, Ioannis and Shreve, Steven E},
  volume={113},
  series={Graduate Texts in Mathematics},
  edition={2},
  year={1991},
  publisher={Springer-Verlag},
  address={New York},
  doi={10.1007/978-1-4612-0949-2}
}

@Book{Protter2005,
  author    = {Philip E. Protter},
  title     = {Stochastic Integration and Differential Equations},
  publisher = {Springer},
  year      = {2005},
  edition   = {2nd},
  address   = {Berlin},
  isbn      = {3540003134},
}

@article{Song2012,
  author    = {Qingshuo Song and Gang George Yin and Chao Zhu},
  title     = {Optimal switching with constraints and utility maximization of an indivisible market},
  journal   = {SIAM Journal on Control and Optimization},
  volume    = {50},
  number    = {2},
  pages     = {629--651},
  year      = {2012},
  doi       = {10.1137/10080823X},
  mrnumber  = {MR2913335?},
  zblnumber = {1335.49041},
}

@article{GuoZhang2004,
  author    = {Xin Guo and Qing Zhang},
  title     = {Closed-form solutions for perpetual {A}merican put options with regime switching},
  journal   = {SIAM Journal on Applied Mathematics},
  volume    = {64},
  number    = {6},
  pages     = {2034--2049},
  year      = {2004},
  doi       = {10.1137/S0036139903426083},
  mrnumber  = {MR2118865?},
  zblnumber = {1061.90082},
}

@article{SongStockbridgeZhu2011,
  author    = {Qingshuo Song and Richard H. Stockbridge and Chao Zhu},
  title     = {On optimal harvesting problems in random environments},
  journal   = {SIAM Journal on Control and Optimization},
  volume    = {49},
  number    = {2},
  pages     = {859--889},
  year      = {2011},
  doi       = {10.1137/100797333},
}

@article{CaoYao2018,
  author    = {Ping Cao and Dacheng Yao},
  title     = {Optimal Drift Rate Control and Impulse Control for a Stochastic Inventory/Production System},
  journal   = {SIAM Journal on Control and Optimization},
  volume    = {56},
  number    = {3},
  pages     = {1856--1883},
  year      = {2018},
}

@book{folland1999real,
  author    = {Folland, Gerald B.},
  title     = {Real Analysis: Modern Techniques and their Applications},
  edition   = {2},
  series    = {Pure and Applied Mathematics},
  publisher = {John Wiley \& Sons},
  address   = {New York},
  year      = {1999},
  isbn      = {0-471-31716-0},
}

@article{DaiYao2013Part2,
  author    = {J. G. Dai and Dacheng Yao},
  title     = {Brownian inventory models with convex holding cost, part 2: Discount-optimal controls},
  journal   = {Stochastic Systems},
  volume    = {3},
  number    = {2},
  pages     = {500--573},
  year      = {2013},
}

@article{NerloveArrow1962,
  author    = {Marc Nerlove and Kenneth J. Arrow},
  title     = {Optimal advertising policy under dynamic conditions},
  journal   = {Economica},
  volume    = {29},
  number    = {114},
  pages     = {129--142},
  year      = {1962},
}

@article{Raman2006,
  author    = {Kalyan Raman},
  title     = {Boundary value problems in stochastic optimal control of advertising},
  journal   = {Automatica},
  volume    = {42},
  number    = {8},
  pages     = {1357--1362},
  year      = {2006},
}

@article{GrossetViscolani2009,
  author    = {Luca Grosset and Bruno Viscolani},
  title     = {Optimal dynamic advertising with an adverse exogenous effect on brand goodwill},
  journal   = {Automatica},
  volume    = {45},
  number    = {4},
  pages     = {863--870},
  year      = {2009},
}

@article{Marinelli2007,
  author    = {Carlo Marinelli},
  title     = {The stochastic goodwill problem},
  journal   = {European Journal of Operational Research},
  volume    = {176},
  number    = {1},
  pages     = {389--404},
  year      = {2007},
}

\newpage

\begin{appendices}
\renewcommand{\thesection}{\Alph{section}}

\maketitle
\section{Proofs in Section \ref{sec:policies}}
\label{app:section-2}

\subsection{Proof of Lemma \ref{lem:J}}
\label{app-sub:lem-J}

    ($a$) We only prove \eqref{eq:J-h} for $J_h(x)$, and the proof for \eqref{eq:J-l} is very similar and is omitted. It follows from the first part of \eqref{eq:J-h-ell-0} that
    \begin{align}
    J_h(x) 
        &= \mathbb{E}_{(x,\mu_h)}\left[ \int_{0}^{\infty}e^{-\beta t}\left(r(X^{h}(t))^{+}-c \right)\,\mathrm{d}t \right] \nonumber\\
        &= r\int_{0}^{\infty}e^{-\beta t} \mathbb{E}_{(x,\mu_h)}\left[ (X^{h}(t))^{+} \right]\dd t -\frac{c}{\beta}\label{eq:J-h-app}
    \end{align}
where we interchanged expectation and integration for the nonnegative payoff term by Tonelli's theorem.
We next calculate $\mathbb{E}_{x}\left[ (X^{h}(t))^{+} \right]$. First we notice that for fixed $t>0$, $X^h(t)$ follows a normal distribution $\mathcal{N}(m_t^h(x), v_t)$ with $m_t^h$ and $v_t$ defined in \eqref{eq:m-v}; see e.g., Example 6.8 in \S5.6 of \cite{karatzas1991brownian}. Then, we have
\begin{align}
    \mathbb{E}_{(x,\mu_h)}\left[ (X^{h}(t))^{+} \right] 
    & = \int_0^{\infty} \mathbb{P}\left(X^{h}(t)>y\right)\,\dd y \nonumber\\
    & = \int_0^{\infty}\left[1 - \Phi\left( \frac{y - m_t^h(x)}{\sqrt{v_t}} \right)\right]\,\dd y\nonumber\\ 
    & = \int_{0}^{\infty} \Phi\left(\frac{m_t^h(x) - y}{\sqrt{v_t}}\right)\,\dd y\nonumber\\
    & = \int_{-\infty}^{m_t^h(x)} \Phi\left(\frac{z}{\sqrt{v_t}}\right)\,\dd z,\label{eq:E-X-h-app}
\end{align}
where the third equality is due to the symmetry of the normal distribution.
Hence, \eqref{eq:J-h} can be obtained immediately by \eqref{eq:J-h-app} and \eqref{eq:E-X-h-app}.

Moreover, by the Cauchy--Schwarz inequality,
\[
\mathbb{E}_{(x,\mu_h)}[(X^h(t))^+]\leq |m_t^h(x)|+\sqrt{v_t}\leq \max\{|x|,|\mu_h|\} +\frac{|\sigma|}{\sqrt{2\theta}}.
\]
Since $\beta>0$, the discounted payoff is integrable. Thus $J_h(x)$ is finite for every $x\in\mathbb{R}$, and the same argument applies to $J_\ell(x)$.

($b$) We only prove the properties for $J_h$, and the proof for $J_{\ell}$ is similar and is omitted. 
Denote the integrand in \eqref{eq:J-h} as $u(x,t)$, i.e.,
\[
u(x,t)=\int_{-\infty}^{m_t^h(x)} e^{-\beta t} \Phi\left(\frac{z}{\sqrt{v_t}}\right)\,\dd z.
\]
Differentiating with respect to $x$ gives 
\begin{align*}
    \frac{\partial u(x,t)}{\partial x}= e^{-(\beta+\theta)t}\Phi\!\left(\frac{m_t^h(x)}{\sqrt{v_t}}\right)
    \quad \text{and}\quad
    \frac{\partial^2 u(x,t)}{\partial x^2}=e^{-(\beta+2\theta)t}\frac{1}{\sqrt{v_t}}\,\phi\!\left(\frac{m_t^h(x)}{\sqrt{v_t}}\right),
\end{align*}
where $\phi=\Phi'$ is the PDF of the standard normal distribution.
The first one is bounded by $e^{-(\beta+\theta)t}\in \mathcal{L}^1((0,\infty))$, and the second by $\frac{e^{-(\beta+2\theta)t}}{\sqrt{2\pi v_t}}\in \mathcal{L}^1((0,\infty))$. Then, the differentiations and continuity of $J_h'(x)$ and $J_h''(x)$ are justified by dominated convergence (see e.g., Theorem 2.27 in \cite{folland1999real}) and thus
\begin{align}
    J_h'(x)&=r\int_0^\infty e^{-(\beta+\theta)t}\Phi\!\left(\frac{m_t^h(x)}{\sqrt{v_t}}\right)\dd t,\label{eq:J-h-derivate}\\
    J_h''(x)&=r\int_0^\infty e^{-(\beta+2\theta)t}\frac{1}{\sqrt{v_t}}\,\phi\!\left(\frac{m_t^h(x)}{\sqrt{v_t}}\right)\dd t.\label{eq:J-h-derivate-2}
\end{align}
Hence, we have proven $J_h\in\mathcal{C}^2(\mathbb{R})$. Also, $J_h'(x)>0$ and $J_h''(x)>0$ for $x\in\mathbb{R}$ can be obtained immediately by \eqref{eq:J-h-derivate}-\eqref{eq:J-h-derivate-2}.

We next check that $J_h$ satisfies ODE \ref{eq:ODE-J-h}. Since
\[
\frac{\dd}{\dd z}\left[z\Phi\left(\frac{z}{\sqrt{v_{t}}}\right)+ \sqrt{v_{t}}\phi\left(\frac{z}{\sqrt{v_{t}}}\right)\right] = \Phi\left(\frac{z}{\sqrt{v_{t}}}\right),
\]
we have 
\begin{align}
    \int_{-\infty}^{m_t^h(x)} \Phi\left(\frac{z}{\sqrt{v_t}}\right)\,\dd z 
    &=\left[z\Phi\left(\frac{z}{\sqrt{v_{t}}}\right)+ \sqrt{v_{t}}\phi\left(\frac{z}{\sqrt{v_{t}}}\right)\right] \Big\vert_{z=-\infty}^{m_t^h(x)}\nonumber\\
    &=  m_t^h(x)\Phi\left(\frac{m_t^h(x)}{\sqrt{v_t}}\right)+\sqrt{v_t}\phi\left(\frac{m_t^h(x)}{\sqrt{v_t}}\right),\label{eq:int=}
\end{align}
where we used 
\begin{align}
\label{eq:two-limits}
    \lim_{z\rightarrow-\infty}z \Phi(z)=\lim_{z\rightarrow-\infty}\frac{\Phi(z)}{z^{-1}}=\lim_{z\rightarrow-\infty}\frac{\phi(z)}{-z^{-2}}=\lim_{z\rightarrow-\infty}\frac{-z^{2}e^{-\frac{z^2}{2}}}{\sqrt{2\pi}}=0,\ \lim_{z\rightarrow-\infty}\phi(z)=0.
\end{align}
Substituting \eqref{eq:int=} into \eqref{eq:J-h} gives 
\begin{align*}
J_h(x)
&=r\int_0^\infty e^{-\beta t}
\left[m_t^h(x)\Phi\left(\frac{m_t^h(x)}{\sqrt{v_t}}\right)+\sqrt{v_t}\,\phi\left(\frac{m_t^h(x)}{\sqrt{v_t}}\right)\right]\dd t
-\frac{c}{\beta},
\end{align*}
which, together with \eqref{eq:J-h-derivate}-\eqref{eq:J-h-derivate-2}, implies that
\begin{align*}
&\Gamma_hJ_h(x)-\beta J_h(x)+rx^+-c\\
&\quad=r\int_0^\infty e^{-\beta t}
\Bigg[
\left(\theta(\mu_h-x)e^{-\theta t}-\beta m_t^h(x)\right)
\Phi\left(\frac{m_t^h(x)}{\sqrt{v_t}}\right)\\
&\qquad+\left(\frac{\sigma^2e^{-2\theta t}}{2\sqrt{v_t}}
-\beta\sqrt{v_t}\right)
\phi\left(\frac{m_t^h(x)}{\sqrt{v_t}}\right)
\Bigg]\dd t+rx^+.
\end{align*}
Observe that $\partial m_t^h(x)/\partial t=\theta(\mu_h-x)e^{-\theta t}$ and $\dd\sqrt{v_t}/\dd t=\frac{\sigma^2e^{-2\theta t}}{2\sqrt{v_t}}$.
Using the chain rule and the identity $\phi'(z)=-z\phi(z)$, we obtain
\begin{align*}
&\frac{\partial}{\partial t}
\left[
m_t^h(x)\Phi\left(\frac{m_t^h(x)}{\sqrt{v_t}}\right)
+\sqrt{v_t}\,\phi\left(\frac{m_t^h(x)}{\sqrt{v_t}}\right)
\right]
\\&\quad=\theta(\mu_h-x)e^{-\theta t}
\Phi\left(\frac{m_t^h(x)}{\sqrt{v_t}}\right)
+\frac{\sigma^2e^{-2\theta t}}{2\sqrt{v_t}}
\phi\left(\frac{m_t^h(x)}{\sqrt{v_t}}\right).
\end{align*}
Consequently,
\begin{align}
&\Gamma_hJ_h(x)-\beta J_h(x)+rx^+-c\nonumber\\
&\quad=r\int_0^\infty
\frac{\partial}{\partial t}
\left\{
e^{-\beta t}
\left[
m_t^h(x)\Phi\left(\frac{m_t^h(x)}{\sqrt{v_t}}\right)
+\sqrt{v_t}\,\phi\left(\frac{m_t^h(x)}{\sqrt{v_t}}\right)
\right]
\right\}dt+rx^+\nonumber\\
&\quad =r\left[
e^{-\beta t}
\left(
m_t^h(x)\Phi\left(\frac{m_t^h(x)}{\sqrt{v_t}}\right)
+\sqrt{v_t}\,\phi\left(\frac{m_t^h(x)}{\sqrt{v_t}}\right)
\right)
\right] \Big\vert_{t=0}^{t=\infty}+rx^+.\label{eq:ODE-partial}
\end{align}
To justify the improper integration by parts, the derivative computed above satisfies
\begin{align*}
&\left|
\frac{\partial}{\partial t}
\left\{
e^{-\beta t}
\left[
m_t^h(x)\Phi\left(\frac{m_t^h(x)}{\sqrt{v_t}}\right)
+\sqrt{v_t}\,\phi\left(\frac{m_t^h(x)}{\sqrt{v_t}}\right)
\right]
\right\}
\right|\\
&\qquad\leq e^{-\beta t}
\left[
\theta|\mu_h-x|e^{-\theta t}
+\frac{\sigma^2e^{-2\theta t}}{2\sqrt{2\pi v_t}}
+\beta\left(
|m_t^h(x)|+\frac{\sqrt{v_t}}{\sqrt{2\pi}}
\right)
\right].
\end{align*}
Since $v_t\sim\sigma^2t$ as $t\downarrow0$, this bound is $O(t^{-1/2})$ near zero and is therefore integrable there. As $t\to\infty$, the quantity in square brackets remains bounded, so the bound is integrable at infinity because $\beta>0$.
Thus the time derivative is absolutely integrable on $(0,\infty)$, and the integration by parts is valid.

If we can prove 
\begin{align}
    \lim_{t\downarrow 0}e^{-\beta t} \left[m_t^h(x)\Phi\left(\frac{m_t^h(x)}{\sqrt{v_t}}\right) +\sqrt{v_t}\,\phi\left(\frac{m_t^h(x)}{\sqrt{v_t}}\right)\right]&=x^+ \quad \text{and}\label{eq:lim-lemma-1}\\
    \lim_{t\uparrow\infty}e^{-\beta t}\left[m_t^h(x)\Phi\left(\frac{m_t^h(x)}{\sqrt{v_t}}\right)+\sqrt{v_t}\,\phi\left(\frac{m_t^h(x)}{\sqrt{v_t}}\right)\right]&=0,\label{eq:lim-lemma-2}
\end{align}
then \eqref{eq:ODE-partial} implies $\Gamma_hJ_h(x)-\beta J_h(x)+rx^+-c=r(0-x^+)+r x^+=0$, proving that $J_h$ satisfies ODE \eqref{eq:ODE-J-h}. 

We next show \eqref{eq:lim-lemma-1} and \eqref{eq:lim-lemma-2}.
If $x>0$, then $m_t^h(x)\to x>0$ and $\sqrt{v_t}\to0$ as $t\downarrow0$, so
\[
\frac{m_t^h(x)}{\sqrt{v_t}}\longrightarrow+\infty,
\qquad
m_t^h(x)\Phi\left(\frac{m_t^h(x)}{\sqrt{v_t}}\right)
\longrightarrow x \qquad\text{as }t\downarrow0.
\]
If $x<0$, then similarly
\[
\frac{m_t^h(x)}{\sqrt{v_t}}\longrightarrow-\infty,
\qquad
m_t^h(x)\Phi\left(\frac{m_t^h(x)}{\sqrt{v_t}}\right)
\longrightarrow0 \qquad\text{as }t\downarrow0.
\]
If $x=0$, the bound $0\leq\Phi\leq1$ gives
\[
\left|
m_t^h(0)\Phi\left(\frac{m_t^h(0)}{\sqrt{v_t}}\right)
\right|
\leq |m_t^h(0)|
=|\mu_h|(1-e^{-\theta t})
\longrightarrow0 \qquad\text{as }t\downarrow0.
\]
Also, we have 
\[
0\leq
\sqrt{v_t}\,\phi\left(\frac{m_t^h(x)}{\sqrt{v_t}}\right)
\leq \frac{\sqrt{v_t}}{\sqrt{2\pi}}
\longrightarrow 0
\qquad\text{as }t\downarrow0.
\]
Hence, \eqref{eq:lim-lemma-1} has been obtained. 

At the upper endpoint, note that $|m_t^h(x)|\leq\max\{|x|,|\mu_h|\}$ and $\sqrt{v_t}\leq\frac{|\sigma|}{\sqrt{2\theta}}$. 
Hence
\begin{align*}
&\left|
e^{-\beta t}
\left[
m_t^h(x)\Phi\left(\frac{m_t^h(x)}{\sqrt{v_t}}\right)
+\sqrt{v_t}\,\phi\left(\frac{m_t^h(x)}{\sqrt{v_t}}\right)
\right]
\right|\\
&\qquad\leq
e^{-\beta t}
\left(
\max\{|x|,|\mu_h|\}
+\frac{|\sigma|}{\sqrt{4\pi\theta}}
\right)
\longrightarrow0
\qquad\text{as }t\to\infty,
\end{align*}
proving \eqref{eq:lim-lemma-2}.

We next prove $J_h\in \mathcal{C}^\infty(\mathbb{R}\setminus\{0\})$. We have already established that $J_h\in \mathcal{C}^2(\mathbb{R})$ and ODE $\Gamma_hJ_h(x)-\beta J_h(x)+rx^+-c=0$ for $x\in\mathbb{R}$. Since $\sigma>0$, the ODE can be rewritten as
\[
J_h''(x)=\frac{2}{\sigma^2}\left[\beta J_h(x)-\theta(\mu_h-x)J_h'(x)-rx^++c\right].
\]
We now prove smoothness on $\mathbb{R}\setminus\{0\}$ by induction. For any integer $k\geq2$, suppose that $J_h\in \mathcal{C}^k(\mathbb{R}\setminus\{0\})$. Since the coefficients and the function $x\mapsto x^+$ are smooth on $\mathbb{R}\setminus\{0\}$, the right-hand side of the preceding identity belongs to $\mathcal{C}^{k-1}(\mathbb{R}\setminus\{0\})$.
Consequently, $J_h''\in \mathcal{C}^{k-1}(\mathbb{R}\setminus\{0\})$, which implies
\[
J_h\in \mathcal{C}^{k+1}(\mathbb{R}\setminus\{0\}).
\]
Starting from the established $J_h\in \mathcal{C}^2(\mathbb{R})$ and applying this argument inductively yields
\[
J_h\in \mathcal{C}^\infty(\mathbb{R}\setminus\{0\}).
\]
Therefore, $J_h\in \mathcal{C}^2(\mathbb{R})\cap \mathcal{C}^\infty(\mathbb{R}\setminus\{0\})$.

It remains to prove that $J_h'$ is Lipschitz continuous on $\mathbb{R}$. Since $\phi(z)\le 1/\sqrt{2\pi}$ for all $z\in\mathbb R$, we obtain
\[
|J_h''(x)|\le \frac{r}{\sqrt{2\pi}}
\int_0^\infty e^{-(\beta+2\theta)t}\frac{1}{\sqrt{v_t}}\,dt
=: M,
\]
for all $x\in\mathbb R$. 
The constant $M$ is finite because
\begin{align*}
    \int_{0}^{\infty} \frac{e^{-(\beta+2\theta)t}}{\sqrt{1-e^{-2\theta t}}} \dd t
    &= \int_{1}^{0} \frac{x^{\frac{\beta}{2\theta}+1}}{\sqrt{1-x}}
    \left(-\frac{1}{2\theta x}\right) \dd x \\
    &= \frac{1}{2\theta}\int_{0}^{1}
    x^{\frac{\beta}{2\theta}}(1-x)^{-\frac{1}{2}}\dd x \\
    &= \frac{1}{2\theta}\mathrm{B}
    \left(\frac{\beta}{2\theta}+1,\frac{1}{2}\right).
\end{align*}
where $\mathrm{B}(\frac{\beta}{2\theta}+1,\frac{1}{2})$ denotes the Beta function with parameters $p=\frac{\beta}{2\theta}+1$ and $q=\frac{1}{2}$, which is finite because both parameters are positive. Therefore, for any $x,y\in\mathbb R$, by the mean value theorem applied to $J_h'$, there exists $\xi$ between $x$ and $y$ such that
\[
|J_h'(x)-J_h'(y)|=|J_h''(\xi)|\cdot |x-y|\le M|x-y|.
\]
Hence $J_h'$ is Lipschitz continuous on $\mathbb R$.

($c$) We only provide the details for the limits of $J_h$ and $J_h'$. The limits of $J_\ell$ and $J_\ell'$ are similar and are omitted. 

First, the definition of $m_t^h(x)=\mu_h+(x-\mu_h)e^{-\theta t}$ shows that $\int_{-\infty}^{m_t^h(x)} \Phi\left(\frac{z}{\sqrt{v_t}}\right)\,\dd z$ is strictly increasing on $x\in\R$ for any $t>0$, thus
\begin{align*}
    \lim_{x\rightarrow\infty}\int_{-\infty}^{m_t^h(x)} \Phi\left(\frac{z}{\sqrt{v_t}}\right)\,\dd z 
    =  \int_{-\infty}^{\infty} \Phi\left(\frac{z}{\sqrt{v_t}}\right)\,\dd z
    \geq \int_0^\infty \Phi(0)\,\dd z
    =\infty.
\end{align*}
Therefore, by the monotone convergence theorem,
\begin{align*}
    \lim_{x\rightarrow\infty}J_{h}(x)&=\lim_{x\rightarrow\infty}\left[r\int_{0}^{\infty}\int_{-\infty}^{m_t^h(x)} e^{-\beta t} \Phi\left(\frac{z}{\sqrt{v_t}}\right)\,\dd z\dd t -\frac{c}{\beta}\right]\\
    &=r\int_{0}^{\infty}\lim_{x\rightarrow\infty}\left[\int_{-\infty}^{m_t^h(x)} e^{-\beta t} \Phi\left(\frac{z}{\sqrt{v_t}}\right)\,\dd z\right]\dd t -\frac{c}{\beta}\\
    &=\infty.
\end{align*}

Second, for all $x\leq 0$ and $t>0$,
\begin{align*}
    \int_{-\infty}^{m_t^h(x)} e^{-\beta t}\Phi\left(\frac{z}{\sqrt{v_t}}\right)\,\dd z
    &\leq e^{-\beta t}\int_{-\infty}^{m_t^h(0)} \Phi\left(\frac{z}{\sqrt{v_t}}\right)\,\dd z\\
    &=e^{-\beta t}\left[m_t^h(0)\Phi\left(\frac{m_t^h(0)}{\sqrt{v_t}}\right)+\sqrt{v_t}\phi\left(\frac{m_t^h(0)}{\sqrt{v_t}}\right)\right]\\
    &\leq e^{-\beta t}\left(\abs{\mu_h}+\frac{\sigma}{\sqrt{4\pi \theta}}\right),
\end{align*}
where the first equality follows from \eqref{eq:int=}, and the second inequality follows from $m^{h}_{t}(0)=\mu_{h}(1-e^{-\theta t})$ and $v_{t}=\frac{\sigma^2}{2\theta}(1-e^{-2\theta t})$.
Then, by the dominated convergence theorem, we obtain
\begin{align*}
    \lim_{x\rightarrow-\infty}J_{h}(x)&=\lim_{x\rightarrow-\infty}\left[r\int_{0}^{\infty}\int_{-\infty}^{m_t^h(x)} e^{-\beta t} \Phi\left(\frac{z}{\sqrt{v_t}}\right)\,\dd z\dd t -\frac{c}{\beta}\right]\\
    &=r\int_{0}^{\infty}\lim_{x\rightarrow-\infty}\left[\int_{-\infty}^{m_t^h(x)} e^{-\beta t} \Phi\left(\frac{z}{\sqrt{v_t}}\right)\,\dd z\right]\dd t -\frac{c}{\beta}\\
    &=r\int_{0}^{\infty}e^{-\beta t}\lim_{x\rightarrow-\infty}\left[m_t^h(x)\Phi\left(\frac{m_t^h(x)}{\sqrt{v_t}}\right)+\sqrt{v_t}\phi\left(\frac{m_t^h(x)}{\sqrt{v_t}}\right)\right]\dd t -\frac{c}{\beta}\\
    &=-\frac{c}{\beta},
\end{align*}
where the third equality follows from \eqref{eq:int=}, and the last equality follows from  \eqref{eq:two-limits}.

We next calculate $\lim_{x\rightarrow-\infty}J_{h}'(x)$ and $\lim_{x\rightarrow\infty}J_{h}'(x)$. 
We have that $e^{-\left(\beta+\theta\right)t}\Phi\left( \frac{m_{t}^{h}(x)}{\sqrt{v_{t}}} \right)\geq0$ is strictly increasing in $x\in\R$, and 
\begin{align*}
    \lim_{x\rightarrow-\infty}e^{-\left(\beta+\theta\right)t}\Phi\left( \frac{m_{t}^{h}(x)}{\sqrt{v_{t}}} \right) = 0 \quad \text{and}\quad 
    \lim_{x\rightarrow\infty}e^{-\left(\beta+\theta\right)t}\Phi\left( \frac{m_{t}^{h}(x)}{\sqrt{v_{t}}} \right) = e^{-\left(\beta+\theta\right)t}.
\end{align*}
Thus, it follows from expression of $J_h'$ in \eqref{eq:J-h-derivate} and the dominated convergence theorem that
\begin{align*}
    &\lim_{x\rightarrow-\infty}J_h'(x)= r\int_{0}^{\infty}\lim_{x\rightarrow-\infty}e^{-\left(\beta+\theta\right)t}\Phi\left( \frac{m_{t}^{h}(x)}{\sqrt{v_{t}}} \right)\dd t=0, \\
    &\lim_{x\rightarrow\infty}J_h'(x)= r\int_{0}^{\infty}\lim_{x\rightarrow\infty}e^{-\left(\beta+\theta\right)t}\Phi\left( \frac{m_{t}^{h}(x)}{\sqrt{v_{t}}} \right)\dd t = \frac{r}{\beta+\theta}.    
\end{align*}

($d$) First, \eqref{eq:J-explicit} can be obtained directly by $J=J_h-J_\ell$ with \eqref{eq:J-h} and \eqref{eq:J-l}.

Second, it follows from $J=J_h-J_\ell$ and $J_h,J_\ell\in\mathcal{C}^{2}(\R)\cap \mathcal{C}^{\infty}(\R\backslash \{0\})$ in part ($b$) that 
\[
J\in\mathcal{C}^{2}(\R)\cap \mathcal{C}^{\infty}(\R\backslash \{0\}).
\]
Since we have proven \eqref{eq:ODE-J-h}-\eqref{eq:ODE-J-l} in part ($b$), the ODE \eqref{eq:ODE-Jh-Jl-h} in Remark \ref{rem:lemma-1} holds, i,e,
\[
\Gamma_{h}J(x) - \beta J(x)-c+\theta(\mu_{h}-\mu_{\ell})J_{\ell}'(x)=0,
\]
Rewriting it as
\[
J''(x)=\frac{2}{\sigma^2}\left[\beta J(x)-\theta(\mu_h-x)J'(x)+c-\theta(\mu_h-\mu_\ell)J_\ell'(x)\right],
\]
whose right-hand side belongs to $\mathcal{C}^{1}(\R)$ because $J\in\mathcal{C}^{2}(\R)$ and $J_{\ell}'\in\mathcal{C}^{1}(\R)$. Thus, the left-hand side $J''\in\mathcal{C}^{1}(\R)$, i.e., $J\in\mathcal{C}^{3}(\R)$ and then $J\in\mathcal{C}^{3}(\R)\cap \mathcal{C}^{\infty}(\R\backslash \{0\})$.

We next prove that $J'(x)>0$ for $x\in\mathbb{R}$. It follows from \eqref{eq:J-explicit} that
\begin{align*}
    J'(x) & = r\int_{0}^{\infty} e^{-\beta t} \left[\Phi\left(\frac{m_t^h(x)}{\sqrt{v_t}}\right)\frac{\partial m_t^h(x)}{\partial x}-\Phi\left(\frac{m_t^\ell(x)}{\sqrt{v_t}}\right)\frac{\partial m_t^{\ell}(x)}{\partial x}\right]\,\dd t \\
    &= r\int_{0}^{\infty} e^{-(\beta+\theta) t} \left[\Phi\left(\frac{m_t^h(x)}{\sqrt{v_t}}\right)-\Phi\left(\frac{m_t^\ell(x)}{\sqrt{v_t}}\right)\right]\,\dd t\\
    &>0,
\end{align*}
where the second equality follows from $\frac{\partial m_t^h(x)}{\partial x}=\frac{\partial m_t^{\ell}(x)}{\partial x}=e^{-\theta t}$, and the inequality follows from that $\Phi\left(\frac{m_t^h(x)}{\sqrt{v_t}}\right)>\Phi\left(\frac{m_t^\ell(x)}{\sqrt{v_t}}\right)$ holds for any $x\in\mathbb{R}$ and $t\in(0,\infty)$ since $m^h_t(x)-m^{\ell}_{t}(x)=(\mu_h-\mu_{\ell})e^{-\theta t}>0$. Hence, we have proven that $J'(x)>0$ in $x\in\mathbb{R}$.

The limits $\lim_{x\to-\infty}J(x)=-\frac{c}{\beta}$ and $\lim_{x\to-\infty}J'(x)=\lim_{x\to\infty}J'(x)=0$ can be obtained immediately by the limits in part ($c$) with noting $J=J_h-J_\ell$. 

It remains to prove $\lim_{x\to\infty}J(x)=\frac{r\theta(\mu_{h}-\mu_{\ell})}{\beta(\beta+\theta)}-\frac{c}{\beta}$.
Since 
\[
\Phi\left(\frac{m_t^{\ell}(x)}{\sqrt{v_t}}\right)\leq \Phi\left(\frac{z}{\sqrt{v_t}}\right)
\leq \Phi\left(\frac{m_t^{h}(x)}{\sqrt{v_t}}\right) \quad 
\text{for $z\in[m_t^{\ell}(x),m_t^h(x)]$},
\]
we have 
\begin{align}
\int_{m_t^{\ell}(x)}^{m_t^h(x)} \Phi\left(\frac{z}{\sqrt{v_t}}\right)\,\dd z
&\geq \int_{m_t^{\ell}(x)}^{m_t^h(x)}  \Phi\left(\frac{m_t^{\ell}(x)}{\sqrt{v_t}}\right)\,\dd z
\nonumber\\
&=  \Phi\left(\frac{m_t^{\ell}(x)}{\sqrt{v_t}}\right)\left[(\mu_h-\mu_{\ell})\left(1-e^{-\theta t}\right)\right]\label{eq:int-geq-app}\\
\int_{m_t^{\ell}(x)}^{m_t^h(x)} \Phi\left(\frac{z}{\sqrt{v_t}}\right)\,\dd z
&\leq \int_{m_t^{\ell}(x)}^{m_t^h(x)}  \Phi\left(\frac{m_t^{h}(x)}{\sqrt{v_t}}\right)\,\dd z
\nonumber\\
&=  \Phi\left(\frac{m_t^{h}(x)}{\sqrt{v_t}}\right)\left[(\mu_h-\mu_{\ell})\left(1-e^{-\theta t}\right)\right]\label{eq:int-leq-app}
\end{align}
where both equalities follow from $m_t^h(x)-m_t^{\ell}(x)=(\mu_h-\mu_{\ell})\left(1-e^{-\theta t}\right)$.
Notice that 
\begin{align*}
    \lim_{x\to\infty}\Phi\left(\frac{m_t^{\ell}(x)}{\sqrt{v_t}}\right)\left[(\mu_h-\mu_{\ell})\left(1-e^{-\theta t}\right)\right]
    &=  \lim_{x\to\infty}\Phi\left(\frac{m_t^{h}(x)}{\sqrt{v_t}}\right)\left[(\mu_h-\mu_{\ell})\left(1-e^{-\theta t}\right)\right]\\
    &=(\mu_h-\mu_{\ell})\left(1-e^{-\theta t}\right),
\end{align*}
which, together with \eqref{eq:int-geq-app}-\eqref{eq:int-leq-app}, imply that
\begin{align}
\label{eq:two-lim-app}
    \lim_{x\to\infty}\int_{m_t^{\ell}(x)}^{m_t^h(x)} \Phi\left(\frac{z}{\sqrt{v_t}}\right)\,\dd z
    = (\mu_h-\mu_{\ell})\left(1-e^{-\theta t}\right).
\end{align}
Further, noticing that
\begin{align*}
    0\leq \int_{m_t^{\ell}(x)}^{m_t^h(x)} e^{-\beta t} \Phi\left(\frac{z}{\sqrt{v_t}}\right)\,\dd z 
    &\leq \int_{m_t^{\ell}(x)}^{m_t^h(x)} e^{-\beta t} \,\dd z
    =e^{-\beta t}\left[(\mu_h-\mu_{\ell})\left(1-e^{-\theta t}\right)\right]
\end{align*}
and 
\[
\int_0^{\infty} e^{-\beta t}\left[(\mu_h-\mu_{\ell})\left(1-e^{-\theta t}\right)\right]\,\dd t
=(\mu_h-\mu_\ell)\left(\frac{1}{\beta}-\frac{1}{\beta+\theta}\right)
<\infty
\]
we, using the dominated convergence theorem, have
\begin{align*}
    \lim_{x\to\infty}J(x) 
    & =  \lim_{x\to\infty}r\int_{0}^{\infty}\int_{m_t^{\ell}(x)}^{m_t^h(x)} e^{-\beta t} \Phi\left(\frac{z}{\sqrt{v_t}}\right)\,\dd z\dd t -\frac{c}{\beta}\\
    &= r\int_{0}^{\infty}e^{-\beta t}\left[\lim_{x\to\infty}\int_{m_t^{\ell}(x)}^{m_t^h(x)}  \Phi\left(\frac{z}{\sqrt{v_t}}\right)\,\dd z\right]\dd t -\frac{c}{\beta}\\
    &=r\int_{0}^{\infty}e^{-\beta t}(\mu_h-\mu_{\ell})\left(1-e^{-\theta t}\right)\dd t -\frac{c}{\beta}\\
    &=\frac{r\theta(\mu_{h}-\mu_{\ell})}{\beta(\beta+\theta)}-\frac{c}{\beta},
\end{align*}
where the third equality follows from \eqref{eq:two-lim-app}.
\qed

\subsection{Proof of Lemma \ref{lem:psi}}
\label{app-sub:lem-psi}

    ($a$) Substituting $\psi_{+}^h(x)$ and $\psi_{-}^h(x)$ into \eqref{eq:ODE-psi-h} immediately shows that both $\psi_{+}^h(x)$ and $\psi_{-}^h(x)$ are two solutions to \eqref{eq:ODE-psi-h}. In fact, they are two fundamental solutions to \eqref{eq:ODE-psi-h}, i.e., any solution to \eqref{eq:ODE-psi-h} is their linear combination; see e.g., page 406 in \cite{ngo2016optimal}. 
    The same argument shows that $\psi_{+}^\ell(x)$ and $\psi_{-}^\ell(x)$ are two fundamental solutions to \eqref{eq:ODE-psi-l}.

($b$) Differentiating \eqref{eq:psi-h-} with respect to $x$ gives 
    \begin{align*}
        (\psi_{-}^h)'(x)=-\frac{\sqrt{2\theta}}{\sigma}\int_{0}^{\infty}t^{\frac{\beta}{\theta}}\exp\left(-\frac{t^{2}}{2}-\frac{\sqrt{2\theta}}{\sigma} (x-\mu_h)t\right)\dd t<0.
    \end{align*}
    The remaining claims follow similarly and are omitted.

($c$)  First, we have 
    \begin{align*}
        \lim_{x\to-\infty}\psi_{-}^h(x)&=\lim_{x\to-\infty}\int_{0}^{\infty}t^{\frac{\beta}{\theta}-1}\exp\left(-\frac{t^{2}}{2}-\frac{\sqrt{2\theta}}{\sigma} (x-\mu_h)t\right)\dd t\\
        &\geq \lim_{x\to-\infty}\int_{1}^{2}t^{\frac{\beta}{\theta}-1}\exp\left(-\frac{t^{2}}{2}-\frac{\sqrt{2\theta}}{\sigma} (x-\mu_h)t\right)\dd t\\
        &=\int_{1}^{2}t^{\frac{\beta}{\theta}-1}\lim_{x\to-\infty}\exp\left(-\frac{t^{2}}{2}-\frac{\sqrt{2\theta}}{\sigma} (x-\mu_h)t\right)\dd t\\
        &=\infty,
    \end{align*}
    where the interchange of the limit and the integral follows from the monotone convergence theorem.
    The proofs of the other limit in \eqref{eq:limit-psi-h-1} and the two limits in \eqref{eq:limit-psi-ell-2} are very similar and are omitted.

    We next prove the first limit in \eqref{eq:limit-psi-h-2}. The proofs of the other limit in \eqref{eq:limit-psi-h-2} and the two limits in \eqref{eq:limit-psi-ell-1} are very similar and are omitted.
    For any $t>0$, we have that $t^{\frac{\beta}{\theta}-1}\exp\left(-\frac{t^{2}}{2}-\frac{\sqrt{2\theta}}{\sigma} (x-\mu_h)t\right)$ is strictly decreasing in $x\in[\mu_h,\infty)$, and 
    \begin{align*}
        \lim_{x\to\infty}t^{\frac{\beta}{\theta}-1}\exp\left(-\frac{t^{2}}{2}-\frac{\sqrt{2\theta}}{\sigma} (x-\mu_h)t\right) = 0.
    \end{align*}
    Furthermore, for any $x\geq \mu_h$ and $t>0$, we have
    \begin{align*}
        0\leq t^{\frac{\beta}{\theta}-1}\exp\left(-\frac{t^{2}}{2}-\frac{\sqrt{2\theta}}{\sigma} (x-\mu_h)t\right) \leq t^{\frac{\beta}{\theta}-1}\exp\left(-\frac{t^{2}}{2}\right).
    \end{align*}
    Since the upper bound satisfies $\int_{0}^{\infty}t^{\frac{\beta}{\theta}-1}\exp\left(-\frac{t^{2}}{2}\right)\dd t<\infty$, by the dominated convergence theorem, we obtain 
    \begin{align*}
        \lim_{x\rightarrow\infty}\psi_{-}^{h}(x) &=\lim_{x\to\infty}\int_{0}^{\infty}t^{\frac{\beta}{\theta}-1}\exp\left(-\frac{t^{2}}{2}-\frac{\sqrt{2\theta}}{\sigma} (x-\mu_h)t\right)\dd t \\
        &=\int_{0}^{\infty}\lim_{x\to\infty}\left[t^{\frac{\beta}{\theta}-1}\exp\left(-\frac{t^{2}}{2}-\frac{\sqrt{2\theta}}{\sigma} (x-\mu_h)t\right)\right]\dd t =0.
    \end{align*}

($d$) It follows from \eqref{eq:psi-h-} that 
    \begin{align*}
        (\psi_{-}^h)''(x)&=\left(\frac{\sqrt{2\theta}}{\sigma}\right)^2\int_{0}^{\infty}t^{\frac{\beta}{\theta}+1}\exp\left(-\frac{t^{2}}{2}-\frac{\sqrt{2\theta}}{\sigma} (x-\mu_h)t\right)\dd t>0 \quad \text{for $x\in\R$, and}\\
        (\psi_{-}^h)'''(x)&=-\left(\frac{\sqrt{2\theta}}{\sigma}\right)^3\int_{0}^{\infty}t^{\frac{\beta}{\theta}+2}\exp\left(-\frac{t^{2}}{2}-\frac{\sqrt{2\theta}}{\sigma} (x-\mu_h)t\right)\dd t<0 \quad \text{for $x\in\R$}.
    \end{align*}
    These inequalities show that $(\psi_{-}^h)''(x)$ is strictly positive and strictly decreasing on $(d,\infty)$ with an upper bound $(\psi_{-}^h)''(d)$. That means $\sup_{x\in(d,\infty)}\abs{(\psi_{-}^h)''(x)}\leq (\psi_{-}^h)''(d)<\infty$. So $(\psi_{-}^h)'(x)$ is Lipschitz continuous on $(d,\infty)$ for any finite $d\in(-\infty,\infty)$. 
    
    The proof for the Lipschitz continuity of $(\psi_+^{\ell})'(x)$ on $(-\infty,u)$ is similar and is omitted.
\qed

\section{Proofs in Section \ref{sec:parameters}}
\label{app:section-4}

\subsection{Proof of Lemma \ref{lem:minimizer}}
\label{app-sub:lem-minimizer}

By differentiating \eqref{eq:ODE-Jh-Jl-h}, \eqref{eq:ODE-psi-h},
and \eqref{eq:ODE-psi-ell-with-h} and combining the resulting
identities, we obtain
\begin{align}
\Gamma_h f_{A,B}(x)-(\beta+\theta)f_{A,B}(x)
&=-\theta(\mu_h-\mu_\ell)\left[J_\ell''(x)+B(\psi_+^\ell)''(x)\right]. \label{eq:ODE-f-h}
\end{align}
Similarly, by differentiating \eqref{eq:ODE-Jh-Jl-ell}, \eqref{eq:ODE-psi-l}, and \eqref{eq:ODE-psi-h-with-ell} and combining the resulting identities, we obtain 
\begin{align}
\Gamma_\ell f_{A,B}(x)-(\beta+\theta)f_{A,B}(x)
&=-\theta(\mu_h-\mu_\ell)\left[J_h''(x)+A(\psi_-^h)''(x)\right].\label{eq:ODE-f-ell}
\end{align}
Since $J\in\mathcal{C}^3(\mathbb{R})$ and $\psi_-^h,\psi_+^\ell\in\mathcal{C}^{\infty}(\mathbb{R})$, we have $f_{A,B}\in\mathcal{C}^2(\mathbb{R})$.  Thus, at any local minimizer $x_{\min}\in\mathbb{R}$, if it exists,
\[
f_{A,B}'(x_{\min})=0,
\qquad
f_{A,B}''(x_{\min})\geq0.
\]

If $B\geq0$, evaluating \eqref{eq:ODE-f-h} at $x_{\min}$ gives
\begin{align*}
(\beta+\theta)f_{A,B}(x_{\min})
&=\frac{\sigma^2}{2}f_{A,B}''(x_{\min})+\theta(\mu_h-\mu_\ell)
\left[
J_\ell''(x_{\min})+B(\psi_+^\ell)''(x_{\min})
\right]
>0,
\end{align*}
where the strict inequality follows from $J_\ell''(x_{\min})>0$, $(\psi_+^\ell)''(x_{\min})>0$, and $\theta(\mu_h-\mu_\ell)>0$.

If $B<0$, the assumption implies $A\geq0$. Evaluating \eqref{eq:ODE-f-ell} at $x_{\min}$ similarly yields
\begin{align*}
(\beta+\theta)f_{A,B}(x_{\min})
&=\frac{\sigma^2}{2}f_{A,B}''(x_{\min})+\theta(\mu_h-\mu_\ell)
\left[
J_h''(x_{\min})+A(\psi_-^h)''(x_{\min})
\right]
>0.
\end{align*}

Since $\beta+\theta>0$, we conclude that $f_{A,B}(x_{\min})>0$ in either case.
\qed

\subsection{Proof of Lemma \ref{lem:A>0,B>0}}
\label{app-sub:lem-A-B}

Let $(d,u,A,B)$ be such a solution. The smooth-fit equations \eqref{eq:optimal-equation-du-2} give
\[
f_{A,B}(d)=f_{A,B}(u)=0.
\]

Suppose first that $A\leq0$. The smooth-fit condition at $d$ (i.e., first equation in \eqref{eq:optimal-equation-du-2}) implies
\[
B=\frac{J'(d)+A(\psi_-^h)'(d)}{(\psi_+^\ell)'(d)}>0,
\]
where the strict inequality follows from $J'(d)>0$, $(\psi_-^h)'(d)<0$, and $(\psi_+^\ell)'(d)>0$; see Lemmas~\ref{lem:J}($d$)
and~\ref{lem:psi}($b$).

The limits in Lemmas~\ref{lem:J}($d$) and~\ref{lem:psi}($c$) yield
\[
\lim_{x\to-\infty}f_{A,B}(x)
=\lim_{x\to-\infty}\left[A(\psi_{-}^{h})'(x)-B(\psi_{+}^{\ell})'(x)+J'(x)\right]
=
\begin{cases}
+\infty, & A<0,\\
0, & A=0.
\end{cases}
\]
Together with $f_{A,B}(d)=f_{A,B}(u)=0$ and $d<u$, these limits imply that $f_{A,B}$ has a nonpositive local minimum in $(-\infty,u)$.
Indeed, if $f_{A,B}$ takes a negative value on this interval, its nonnegative limit at $-\infty$ and its zero value at $u$ ensure that a negative minimum is attained at an interior point of a suitable compact subinterval. Otherwise, $f_{A,B}\geq0$ on $(-\infty,u)$, and $d$ is a local minimizer by definition with value zero. Since $B>0$, either possibility contradicts Lemma~\ref{lem:minimizer}. Hence $A>0$.

Suppose next that $B\leq0$. Since $A>0$ has already been established, Lemma~\ref{lem:minimizer} again applies. The limits now give
\[
\lim_{x\to\infty}f_{A,B}(x)
=\lim_{x\to\infty}\left[A(\psi_{-}^{h})'(x)-B(\psi_{+}^{\ell})'(x)+J'(x)\right]=
\begin{cases}
+\infty, & B<0,\\
0, & B=0.
\end{cases}
\]
Applying the same argument on $(d,\infty)$, using $f_{A,B}(d)=0$ and the interior zero $f_{A,B}(u)=0$, shows that $f_{A,B}$ has a nonpositive local minimum in this interval. This again contradicts Lemma~\ref{lem:minimizer}.
Therefore $B>0$, completing the proof.
\qed

\subsection{Proof of Lemma \ref{lem:fAB}}
\label{app-sub:lem-fAB}

All parameter pairs considered below satisfy $A,B\geq0$. By Lemma~\ref{lem:minimizer}, $f_{A,B}$ cannot have a local
minimum with a nonpositive value. Moreover, $f_{A,B}\in\mathcal{C}^2(\mathbb{R})$.

We first show that $f_{A,B}$ is strictly positive between any two distinct zeros. Suppose that
\[
f_{A,B}(z_1)=f_{A,B}(z_2)=0, \qquad z_1<z_2.
\]
If $f_{A,B}$ takes a negative value in $(z_1,z_2)$, then its minimum on the compact interval $[z_1,z_2]$ is negative and is attained at an interior point. This contradicts Lemma~\ref{lem:minimizer}. Hence $f_{A,B}\geq0$ on $[z_1,z_2]$. If it vanishes at any point of $(z_1,z_2)$, that point is a local minimizer with value zero, giving the same contradiction. Therefore,
\[
f_{A,B}(x)>0,\qquad z_1<x<z_2.
\]
In particular, $f_{A,B}$ has at most two distinct zeros: a third zero would lie strictly between two other zeros, contradicting the preceding conclusion. We now consider the cases $A,B>0$, $A>0$ and $B=0$, and $A=0$ and $B>0$.

Suppose that $A>0$, $B>0$, and $\sup_{x\in\mathbb{R}}f_{A,B}(x)>0$. By the limits in Lemmas~\ref{lem:J}($d$) and~\ref{lem:psi}($c$),
\begin{align*}
    \lim_{x\to-\infty}f_{A,B}(x)&=\lim_{x\to-\infty}\left[A(\psi_{-}^{h})'(x)-B(\psi_{+}^{\ell})'(x)+J'(x)\right]=-\infty,\\
    \lim_{x\to\infty}f_{A,B}(x)&=\lim_{x\to\infty}\left[A(\psi_{-}^{h})'(x)-B(\psi_{+}^{\ell})'(x)+J'(x)\right]=-\infty.
\end{align*}
Choose $x_1\in\mathbb{R}$ such that $f_{A,B}(x_1)>0$. The intermediate value theorem gives a zero to the left of $x_1$ and another zero to its right. Since there are at most two distinct zeros, these are the unique zeros, denoted by $d(A,B)$ and $u(A,B)$, with $d(A,B)<x_1<u(A,B)$.

The observation above gives
\[
f_{A,B}(x)>0,\qquad x\in(d(A,B),u(A,B)).
\]
On each of the intervals $(-\infty,d(A,B))$ and $(u(A,B),\infty)$, the function is continuous and has no zeros, so it has a constant strict sign. The limits at $-\infty$ and $+\infty$ show that this sign is negative on both intervals. This proves \eqref{eq:fAB-sign} with $A>0$ and $B>0$.

Suppose that $A>0$, $B=0$, and $\sup_{x\in\mathbb{R}}f_{A,0}(x)>0$. The limits in Lemmas~\ref{lem:J}($d$) and~\ref{lem:psi}($c$) now give
\[
\lim_{x\to-\infty}f_{A,0}(x)=-\infty,
\qquad
\lim_{x\to\infty}f_{A,0}(x)=0.
\]
Choose $x_2\in\mathbb{R}$ such that $f_{A,0}(x_2)>0$. By the intermediate value theorem, $f_{A,0}$ has at least one zero to the left of $x_2$.

We show that it cannot have two distinct zeros. Suppose, to the contrary, that $z_3<z_4$ are its two zeros. Then
\[
f_{A,0}(x)>0,\qquad z_3<x<z_4.
\]
Since there are no further zeros, continuity implies that $f_{A,0}$ has a constant strict sign on $(z_4,\infty)$. This sign cannot be positive, because otherwise $z_4$ would be a local minimizer with value zero. Consequently,
\[
f_{A,0}(x)<0,\qquad x>z_4.
\]
Choose $y>z_4$. Since $f_{A,0}(y)<0$ and $f_{A,0}(x)\to0$ as $x\to+\infty$, there exists $R>y$ such that
$f_{A,0}(R)>f_{A,0}(y)$. Also, $f_{A,0}(z_4)=0>f_{A,0}(y)$. Thus the minimum of $f_{A,0}$ on $[z_4,R]$ is negative and is attained at an interior point. This contradicts Lemma~\ref{lem:minimizer}.
Therefore, $f_{A,0}$ has exactly one zero, denoted by $d(A,0)$. 
Since $d(A,0)<x_2$, continuity, uniqueness of the zero, and the limit at $-\infty$ and $\infty$ imply
\[
f_{A,0}(x)<0\quad\text{for }x<d(A,0),
\qquad
f_{A,0}(x)>0\quad\text{for }x>d(A,0).
\]
Setting $u(A,0)=\infty$, we obtain \eqref{eq:fAB-sign} for $A>0$ and $B=0$.

For $A=0$ and $B>0$, the corresponding limits are
\[
\lim_{x\to-\infty}f_{0,B}(x)=0,
\qquad
\lim_{x\to+\infty}f_{0,B}(x)=-\infty.
\]
An analogous argument shows that $f_{0,B}$ has exactly one zero, denoted by $u(0,B)$, and that
\[
f_{0,B}(x)>0
\quad\text{for }x<u(0,B),
\qquad
f_{0,B}(x)<0
\quad\text{for }x>u(0,B).
\]
Setting $d(0,B)=-\infty$  completes the proof of \eqref{eq:fAB-sign} for $A=0$ and $B>0$.

It remains to establish the derivative inequalities in \eqref{eq:fAB-lower-slope} and \eqref{eq:fAB-upper-slope}. The sign pattern and differentiability imply
\[
f_{A,B}'\bigl(d(A,B)\bigr)\geq0
\quad\text{if }A>0,
\qquad
f_{A,B}'\bigl(u(A,B)\bigr)\leq0
\quad\text{if }B>0.
\]
We show that neither derivative can vanish.

Let $z$ be any finite zero identified above, and suppose that $f_{A,B}'(z)=0$. Since $f_{A,B}(z)=0$, evaluating \eqref{eq:ODE-f-h} at $z$ yields
\[
\frac{\sigma^2}{2}f_{A,B}''(z)
=
-\theta(\mu_h-\mu_\ell)
\left[
J_\ell''(z)+B(\psi_+^\ell)''(z)
\right]
<0.
\]
Here we have used $B\geq0$, $J_\ell''(z)>0$, $(\psi_+^\ell)''(z)>0$, and $\theta(\mu_h-\mu_\ell)>0$. Thus $z$ would be a strict local maximum with value zero. This contradicts the strict positivity of $f_{A,B}$ on the side of $z$ lying inside $(d(A,B),u(A,B))$. Therefore the derivative at each finite zero is nonzero, proving the inequalities in
\eqref{eq:fAB-lower-slope} and \eqref{eq:fAB-upper-slope}.
\qed

\subsection{Proof of Lemma \ref{lem:S-properties}}
\label{app-sub:lem-S}

Since $(\psi_-^h)'<0$, $(\psi_+^\ell)'>0$, and $J'>0$, we have
\[
0\leq \bigl(f_{A,B}(x)\bigr)^+\leq J'(x)
\qquad
\text{for all }x\in\mathbb{R},\ A,B\geq0.
\]
Furthermore, Lemma \ref{lem:J} gives
\[
\int_{\mathbb{R}}J'(x)\,\mathrm{d}x
= \lim_{x\to\infty}J(x)-\lim_{x\to-\infty}J(x)
= \frac{r\theta(\mu_h-\mu_\ell)}{\beta(\beta+\theta)}
<\infty.
\]
Thus $S$ is finite everywhere on $[0,\infty)^2$. Since $f_{A,B}(x)$ depends continuously on $(A,B)$ for every fixed $x$, the dominated convergence theorem also implies that $S$ is continuous.

Because $f_{0,0}=J'>0$, the stated formula for $S(0,0)$ follows immediately. For every fixed $x\in\mathbb{R}$,
\[
\lim_{A\to\infty}f_{A,0}(x)=-\infty,
\qquad
\lim_{B\to\infty}f_{0,B}(x)=-\infty.
\]
Applying dominated convergence once again yields
\[
\lim_{A\to\infty}S(A,0) = \lim_{B\to\infty}S(0,B) = 0.
\]

Finally, the stated inequalities between the parameters imply that $f_{A_2,B_2}(x)<f_{A_1,B_1}(x)$ for every $x\in\mathbb{R}$. 
Consequently,
\[
\bigl(f_{A_2,B_2}(x)\bigr)^+
\leq
\bigl(f_{A_1,B_1}(x)\bigr)^+.
\]
If $S(A_1,B_1)>0$, continuity implies that $f_{A_1,B_1}$ is positive on a nonempty open interval. On that interval, the preceding inequality is strict. Integration therefore gives
\[
S(A_2,B_2)<S(A_1,B_1).
\]
\qed

\subsection{Proof of Lemma \ref{lem:area-level}}
\label{app-sub:lem-area}

Fix $\kappa\in(0,\frac{r\theta(\mu_h-\mu_\ell)}{\beta(\beta+\theta)})$. By Lemma \ref{lem:S-properties},
\[
S(0,0)=\frac{r\theta(\mu_h-\mu_\ell)}{\beta(\beta+\theta)}>\kappa,
\qquad
\lim_{A\to\infty}S(A,0)
=\lim_{B\to\infty}S(0,B)
=0.
\]
Continuity, the intermediate value theorem, and the strict comparison property of $S$ therefore give unique
$\bar{A}_\kappa,\bar{B}_\kappa>0$ such that
\[
S(\bar{A}_\kappa,0)=S(0,\bar{B}_\kappa)=\kappa.
\]

For each $B\in[0,\bar B_\kappa]$, monotonicity of $S$ from \eqref{eq:S-monotone} gives
\[
S(0,B)\geq S(0,\bar{B}_\kappa)=\kappa,\qquad S(\bar A_\kappa,B)\leq S(\bar{A}_\kappa,0)=\kappa.
\]
Hence there exists $A_\kappa(B)\in[0,\bar{A}_\kappa]$ satisfying $S(A_\kappa(B),B)=\kappa$. 
This solution is unique among all $A\geq0$.
Indeed, if $0\leq A_1<A_2$ were two solutions, then
$S(A_1,B)=\kappa>0$, and \eqref{eq:S-monotone} would give
$S(A_2,B)<S(A_1,B)$, a contradiction.
Uniqueness and the strict comparison property of $S$ also imply that
\[
A_\kappa(0)=\bar A_\kappa, \qquad A_\kappa(\bar B_\kappa)=0,
\]
and
\[
0<A_\kappa(B)<\bar{A}_\kappa\qquad\text{for }0<B<\bar{B}_\kappa.
\]

To prove continuity, let $B_n\to B$ in $[0,\bar B_\kappa]$. Every convergent subsequence of $A_\kappa(B_n)$ has a limit $a\in[0,\bar A_\kappa]$ satisfying
\[
S(a,B)=\kappa
\]
by continuity of $S$. Uniqueness implies $a=A_\kappa(B)$. Compactness of $[0,\bar A_\kappa]$ therefore yields $A_\kappa(B_n)\to A_\kappa(B)$, proving continuity.

Now fix $B\in(0,\bar{B}_\kappa)$. Since $A_\kappa(B)>0$ and $S(A_\kappa(B),B)>0$, Lemma \ref{lem:fAB} gives two finite zeros at which the derivatives of $f_{A_\kappa(B),B}$ are nonzero. Together with $(x,A,B)\mapsto f_{A,B}(x)$ is continuously differentiable, the implicit function theorem therefore implies that $d(A,B), u(A,B)\in\mathcal{C}^1(O_{A_\kappa(B),B})$,  where $O_{A_\kappa(B),B}$ is a neighborhood of $(A_\kappa(B),B)$. By continuity of $S$, $O_{A_\kappa(B),B}$ may be chosen so that $A,B>0$ and $S(A,B)>0$ throughout it. Moreover, it holds that $S(A,B)\in\mathcal{C}^1(O_{A_\kappa(B),B})$.

For any $(A,B)\in\mathcal{C}^1(O_{A_\kappa(B),B})$, differentiation of $S(A,B)=F_{A,B}(u(A,B))-F_{A,B}(d(A,B))$ and the identities $f_{A,B}(d(A,B))=f_{A,B}(u(A,B))=0$ give 
\begin{align*}
\frac{\partial S}{\partial A}(A,B)&=\psi_-^h(u(A,B))-\psi_-^h(d(A,B))<0, \\
\frac{\partial S}{\partial B}(A,B)&=\psi_+^\ell(d(A,B))-\psi_+^\ell(u(A,B))<0.
\end{align*}
The implicit function theorem applied to $S(A,B)=\kappa$ at $(A_{\kappa}(B),B)$, together with uniqueness, therefore gives $A_\kappa\in\mathcal{C}^1((0,\bar{B}_\kappa))$ and
\begin{equation*}
A_\kappa'(B)
=-\frac{\psi_+^\ell(u(A_\kappa(B),B))-\psi_+^\ell(d(A_\kappa(B),B))}{\psi_-^h(d(A_\kappa(B),B))-\psi_-^h(u(A_\kappa(B),B))}
<0,
\qquad 0<B<\bar B_\kappa,
\end{equation*}
where the inequality follows from the monotonicity of $\psi_+^\ell$ and $\psi_-^h$ in Lemma \ref{lem:psi} ($b$). Continuity on the closed interval and the mean value theorem now imply that $A_\kappa$ is strictly decreasing on $[0,\bar{B}_\kappa]$.

Finally, if $B>\bar B_\kappa$, then 
\[
S(A,B)\leq S(0,B)<S(0,\bar{B}_\kappa)=\kappa\qquad\text{for every }A\geq0.
\]
Thus every nonnegative solution of $S(A,B)=\kappa$ has $B\in[0,\bar{B}_\kappa]$ and, by uniqueness, $A=A_\kappa(B)$.
\qed

\subsection{Proof of Lemma \ref{lem:level-matching}}
\label{app-sub:lem-level}

Set $\kappa=\kappa_2-\kappa_1$. The assumptions imply $0<\kappa<\frac{r\theta(\mu_h-\mu_\ell)}{\beta(\beta+\theta)}$.
Let $\bar{A}_\kappa$, $\bar{B}_\kappa$, and $A_\kappa$ be as in Lemma \ref{lem:area-level}. For $B\in[0,\bar{B}_\kappa]$, define
\[
G_\kappa(B)=F_{A_\kappa(B),B}\bigl(d(A_\kappa(B),B)\bigr).
\]

We begin by proving the continuity of $G_\kappa$ on $[0,\bar{B}_\kappa]$. Fix $B_0\in[0,\bar{B}_\kappa]$.
Since $S(A_\kappa(B_0),B_0)=\kappa>0$, there exists $x_0\in\mathbb{R}$ such that
\[
f_{A_\kappa(B_0),B_0}(x_0)>0.
\]
Continuity of $A_\kappa$ implies that this inequality continues to hold for $B$ sufficiently close to $B_0$. By Lemma \ref{lem:fAB},
\begin{align*}
G_\kappa(B)&=F_{A_\kappa(B),B}(x_0)-\int_{d(A_\kappa(B),B)}^{x_0}f_{A_\kappa(B),B}(x)\,\mathrm{d}x\\
&=F_{A_\kappa(B),B}(x_0)-\int_{-\infty}^{x_0}\bigl(f_{A_\kappa(B),B}(x)\bigr)^+\,\mathrm{d}x.
\end{align*}
The first term is continuous in $B$. For the integral, the bound
\[
0\leq \bigl(f_{A_\kappa(B),B}(x)\bigr)^+\leq J'(x),
\qquad
J'\in \mathcal{L}^1(\mathbb{R}),
\]
allows us to apply the dominated convergence theorem. 
Thus $G_\kappa$ is continuous at $B_0$, including when $B_0=0$ or $B_0=\bar B_\kappa$.

For $0<B<\bar B_\kappa$, Lemma \ref{lem:fAB} and Lemma \ref{lem:area-level} show that $d(A,B)$ and $A_{\kappa}(B)$ are differentiable at $(A_{\kappa}(B),B)$. Using $f_{A_\kappa(B),B}(d(A_\kappa(B),B))=0$, differentiation gives
\[
G_\kappa'(B)=A_\kappa'(B)\psi_-^h(d(A_\kappa(B),B))-\psi_+^\ell(d(A_\kappa(B),B))<0,
\]
because $A_\kappa'(B)<0$ (see Lemma \ref{lem:area-level}) and both $\psi_-^h$ and $\psi_+^\ell$ are strictly positive (see Lemma \ref{lem:psi}). Together with continuity on $[0,\bar B_\kappa]$, the mean value theorem implies that $G_\kappa$ is strictly decreasing on this closed interval.

At $B=0$, it follows from Lemma \ref{lem:fAB} that $u(\bar{A}_\kappa,0)=\infty$ and then
\[
F_{\bar A_\kappa,0}(u(\bar{A}_\kappa,0))=\lim_{x\to\infty}F_{\bar{A}_\kappa,0}(x)
=\lim_{x\to\infty}\left[\bar{A}_\kappa\psi_{-}^{h}(x)+J(x)\right]
=\frac{r\theta(\mu_h-\mu_\ell)}{\beta(\beta+\theta)}-\frac{c}{\beta},
\]
where the last equality follows from \eqref{eq:limit-psi-h-2} and \eqref{eq:lim-J}.
Therefore,
\[
\kappa=S(\bar{A}_\kappa,0)
=F_{\bar A_\kappa,0}(u(\bar{A}_\kappa,0))-F_{\bar A_\kappa,0}(d(\bar{A}_\kappa,0))
=\frac{r\theta(\mu_h-\mu_\ell)}{\beta(\beta+\theta)}-\frac{c}{\beta}-G_\kappa(0),
\]
so
\begin{align}
\label{eq:Gk-lower-bound}
G_\kappa(0)=\frac{r\theta(\mu_h-\mu_\ell)}{\beta(\beta+\theta)}-\frac{c}{\beta}-\kappa.
\end{align}

At $B=\bar{B}_\kappa$, we have $A_\kappa(\bar{B}_\kappa)=0$ (see \eqref{eq:A-kappa-bound}) and $d(A_\kappa(\bar{B}_\kappa),\bar{B}_\kappa)=d(0,\bar{B}_\kappa)=-\infty$ (see Lemma \ref{lem:fAB}).
Thus, 
\begin{align}
\nonumber
G_\kappa(\bar{B}_\kappa)
&=F_{A_\kappa(\bar{B}_\kappa),\bar{B}_\kappa}
(d(A_\kappa(\bar{B}_\kappa),\bar{B}_\kappa)) \\
&=\lim_{x\to-\infty}F_{0,\bar{B}_\kappa}(x) \nonumber\\
&=\lim_{x\to-\infty}
\left[-\bar{B}_\kappa \psi_+^{\ell}(x)+J(x)\right]
=-\frac{c}{\beta},\label{eq:Gk-upper-bound}
\end{align}
where the last equality follows from $\lim_{x\to-\infty}\psi_+^{\ell}(x)=0$ and $\lim_{x\to-\infty}J(x)=-\frac{c}{\beta}$; see \eqref{eq:limit-psi-ell-1} and \eqref{eq:lim-J}.

Since $\kappa=\kappa_2-\kappa_1$, the assumptions on $\kappa_1$ and $\kappa_2$ give
\[
-\frac{c}{\beta}\leq \kappa_1\leq \frac{r\theta(\mu_h-\mu_\ell)}{\beta(\beta+\theta)}-\frac{c}{\beta}-\kappa.
\]
Continuity and strict monotonicity of $G_\kappa$, together with \eqref{eq:Gk-lower-bound} and \eqref{eq:Gk-upper-bound}, 
therefore yield a unique $B\in[0,\bar B_\kappa]$ such that $G_\kappa(B)=\kappa_1$.
Then
\[
F_{A_\kappa(B),B}\bigl(d(A_\kappa(B),B)\bigr)=\kappa_1
\]
and
\[
F_{A_\kappa(B),B}\bigl(u(A_\kappa(B),B)\bigr)
=F_{A_\kappa(B),B}\bigl(d(A_\kappa(B),B)\bigr)+S(A_\kappa(B),B)
=\kappa_1+\kappa
=\kappa_2.
\]

Conversely, let $(A,B)$ be any nonnegative pair with
$A+B>0$ and $S(A,B)>0$ satisfying \eqref{eq:F-kappa1-kappa2}.
The area identity gives
\[
S(A,B)
=F_{A,B}\bigl(u(A,B)\bigr)-F_{A,B}\bigl(d(A,B)\bigr)
=\kappa_2-\kappa_1=\kappa.
\]
Lemma~\ref{lem:area-level} therefore implies
$B\in[0,\bar B_\kappa]$ and $A=A_\kappa(B)$.
The uniqueness of the solution to $G_\kappa(B)=\kappa_1$ then proves uniqueness of the pair $(A,B)$.

Finally, strict monotonicity of $G_\kappa$ gives 
\[
\kappa_1=-\frac{c}{\beta}\iff B=\bar B_\kappa
\iff A=A_{\kappa}(B)=0.
\]
Similarly,
\[
\kappa_2=\frac{r\theta(\mu_h-\mu_\ell)}{\beta(\beta+\theta)}-\frac{c}{\beta}
\iff \kappa_1=\frac{r\theta(\mu_h-\mu_\ell)}{\beta(\beta+\theta)}-\frac{c}{\beta}-\kappa
\iff B=0.
\]
The assertions about finiteness of $d(A,B)$ and $u(A,B)$ follow from Lemma \ref{lem:fAB}.
\qed

\section{Proofs in Section \ref{sec:optimality}}
\label{app:section-5}

\subsection{Proof of Lemma \ref{lem:bound}}
\label{app-sub:lem-bound}
    Under any admissible policy, $X=\{X(t),t\geq0\}$ is given by $\dd X(t) = \theta(\mu(t)-X(t))\dd t + \sigma \dd B(t)$ with $X(0)=x$, which implies
    \begin{align*}
        X(t) = e^{-\theta t}x + \int_0^t \theta e^{-\theta(t-s)}\mu(s)\dd s + \sigma \int_0^t e^{-\theta(t-s)}\dd B(s),\quad t\geq0.
    \end{align*}
    Then, 
    \begin{align}
        &\E_{(x,\mu)}\left[(X(t))^2\right]\nonumber\\ 
        &\quad\leq 3\E_{(x,\mu)}\left[\left(e^{-\theta t}x\right)^2 + \left(\int_0^t \theta e^{-\theta(t-s)}\mu(s)\dd s\right)^2 + \left(\sigma \int_0^t e^{-\theta(t-s)}\dd B(s)\right)^2\right], \label{eq:E-bound}
    \end{align}
    where we used $(a+b+c)^2\leq 3(a^2+b^2+c^2)$. For any $t\geq0$, we have $\left(e^{-\theta t}x\right)^2\leq x^2$, 
    \[
    \left(\int_0^t \theta e^{-\theta(t-s)}\mu(s)\dd s\right)^2 \leq \bar{\mu}^{2}\left(\int_0^t \theta e^{-\theta(t-s)}\dd s\right)^2
        =\bar{\mu}^{2}\left(1-e^{-\theta t}\right)^2\leq \bar{\mu}^{2},
    \]
    and 
    \begin{align*}
        \E_{(x,\mu)}\left[\left(\sigma \int_0^t e^{-\theta(t-s)}\dd B(s)\right)^2\right]
        &= \sigma^2 \E_{(x,\mu)}\left[\int_0^t e^{-2\theta(t-s)}\dd s\right] \\
        &= \frac{\sigma^2}{2\theta}\left[1-e^{-2\theta t}\right]\leq  \frac{\sigma^2}{2\theta},
    \end{align*}
    where the first equality follows from the It\^{o} isometry. 
    
    Substituting the three inequalities above into \eqref{eq:E-bound} gives \eqref{eq:bound-lemma}.
\qed

\end{appendices}
\end{document}